\documentclass[12pt]{amsart}
\usepackage{a4wide,enumerate,color,graphicx}
\usepackage{amsmath}
\usepackage{scrextend} % for labeling environment
\allowdisplaybreaks

\let\pa\partial  
\let\na\nabla  
\let\eps\varepsilon  

\newcommand{\N}{{\mathbb N}}  
\newcommand{\R}{{\mathbb R}} 
\newcommand{\diver}{\operatorname{div}}

\newcommand{\dom}{{{\mathbb T}^d}}

\newtheorem{theorem}{Theorem}
\newtheorem{lemma}[theorem]{Lemma}
\newtheorem{proposition}[theorem]{Proposition}
\newtheorem{remark}[theorem]{Remark}

\begin{document}  

\title[Nonlocal cross-diffusion systems]{Nonlocal 
cross-diffusion systems \\ for multi-species populations and networks}

\author[A. J\"ungel]{Ansgar J\"ungel}
\address{Institute for Analysis and Scientific Computing, Vienna University of  
	Technology, Wiedner Hauptstra\ss e 8--10, 1040 Wien, Austria}
\email{juengel@tuwien.ac.at} 

\author[S. Portisch]{Stefan Portisch}
\address{Institute for Analysis and Scientific Computing, Vienna University of  
	Technology, Wiedner Hauptstra\ss e 8--10, 1040 Wien, Austria}
\email{stefan.portisch@tuwien.ac.at} 

\author[A. Zurek]{Antoine Zurek}
\address{Institute for Analysis and Scientific Computing, Vienna University of  
	Technology, Wiedner Hauptstra\ss e 8--10, 1040 Wien, Austria}
\email{antoine.zurek@tuwien.ac.at}

\date{\today}

\thanks{The authors have been partially supported by the Austrian Science Fund (FWF), 
grants P30000, P33010, F65, and W1245, and by the multilateral project of the
Austrian Agency for International Cooperation in Education and Research
(OeAD), grant MULT 11/2020. This work has received funding from the European 
Research Council (ERC) under the European Union's Horizon 2020 research and 
innovation programme, ERC Advanced Grant no.~101018153.} 

\begin{abstract}
Nonlocal cross-diffusion systems on the torus, arising
in population dynamics and neuroscience, are analyzed. The global existence
of weak solutions, the weak-strong uniqueness, and the
localization limit are proved. The kernels are assumed to be 
positive definite and in detailed balance. The proofs are based on entropy estimates
coming from Shannon-type and Rao-type entropies, while the weak-strong uniqueness 
result follows from the relative entropy method. The existence and uniqueness
theorems hold for nondifferentiable kernels. The associated local cross-diffusion
system, derived in the localization limit, is also discussed.
\end{abstract}

% \paragraph{Keywords:}  
\keywords{Cross diffusion, neural network dynamics, entropy method, localization limit,
global existence of solutions, weak-strong uniqueness.}  
 
% \paragraph{AMS classification:}  
\subjclass[2000]{35K40, 35K55, 35Q92, 68T07, 92B20.}

\maketitle

%%%%%%%%%%%%%%%%%%%%%%%%%%%%%%%%%%%%%%%%%%%%%%%%%%%%%%%%%%%%%%%%%%%%%%%%%%%%%%%

\section{Introduction}

In this paper, we consider the following nonlocal cross-diffusion system:
\begin{equation}\label{1.eq}
  \pa_t u_i - \sigma\Delta u_i  = \diver(u_i\na p_i[u]), \quad t>0, \quad
	u_i(0)=u_i^0\quad\mbox{in }\dom,\ i=1,\ldots,n, 
\end{equation}
where $\sigma>0$ is the diffusion coefficient, $\dom$ is
the $d$-dimensional torus ($d\ge 1$) and $p_i[u]$ is a nonlocal operator given by
\begin{equation}\label{1.nonloc}
   p_i[u](x) = \sum_{j=1}^n\int_\dom K_{ij}(x-y)u_j(y)dy, \quad i=1,\ldots,n,
\end{equation}
$K_{ij}:\dom\to\R$ are the kernel functions (extended periodically to $\R^d$), 
and $u=(u_1,\ldots,u_n)$ is the solution vector. 

Following \cite{GHLP21,PL19}, 
in the case $K_{ij} = a_{ij} K$ with $a_{ij} \in \R$, this model describes the 
dynamics of a population with $n$ species. 
Here, each species can detect other species over a spatial neighborhood
by nonlocal sensing, described by the kernel function $K$.
The coefficient $a_{ij}$ is a measure of the strength of attraction
(if $a_{ij} <0$) or repulsion (if $a_{ij} > 0$)
of the $i$th species to/from the $j$th species.
A typical choice of $K$ is the characteristic function $\mathrm{1}_{B}$ 
of a ball $B$ centered at the origin.
The authors of \cite{GHLP21} proved the local existence of a unique 
(strong) solution to \eqref{1.eq}--\eqref{1.nonloc} for $d\geq 1$
under the condition that $K$ is twice differentiable, 
and this solution can be extended globally in one space dimension.
However, the condition on $K$ excludes the case $K=\mathrm{1}_B$.
One objective of this paper is to extend the results of \cite{GHLP21}
by proving the global existence of weak solutions to \eqref{1.eq}--\eqref{1.nonloc}
and a weak-strong uniqueness result in any space dimension and for 
nondifferentiable kernels $K_{ij}$. As in \cite{GHLP21}, we consider equations
\eqref{1.eq} on the torus; see Remark \ref{rem.domain} for the whole space case
or bounded domains.

Another motivation of the present research comes from the work \cite{CDJ19}, 
where the model \eqref{1.eq}--\eqref{1.nonloc} was rigorously derived from
interacting many-particle systems in a mean-field-type limit. As a by-product
of the limit, the local existence of smooth solutions to \eqref{1.eq}--\eqref{1.nonloc}
could be shown under the assumption that $K_{ij}$ is smooth. Under the same condition,
the so-called localization limit was proved, i.e., if $K_{ij}$ converges to the 
delta distribution times some factor $a_{ij}$, the solution to the nonlocal 
system \eqref{1.eq}--\eqref{1.nonloc} converges to a solution to the model
\eqref{1.eq} with
\begin{equation}\label{1.loc}
  p_i[u] = \sum_{j=1}^n a_{ij}u_j, \quad i=1,\ldots,n.
\end{equation} 
We note that the local system was first introduced in \cite{GaVe19} 
in the case of two species.
In this paper, we generalize the results of \cite{CDJ19}
by imposing ``minimal'' conditions
on the initial datum $u^0$ and the kernels $K_{ij}$.

A third motivation comes from neuroscience. Indeed, following 
\cite{BFFT12,GRT17}, we are interested in the study of deterministic nonlocal 
models of the form \eqref{1.eq}--\eqref{1.nonloc} obtained as the mean-field limit 
of stochastic systems describing the evolution of the states of neurons belonging
to different populations. When the number of neurons becomes very large, the
solutions of the generalized Hodgkin--Huxley model of \cite{BFFT12} can be described
in the mean-field limit by a probability distribution $u_i$ for the $i$th species,
which solves the McKean--Vlasov--Focker--Planck equation of the type
\begin{equation}\label{1.McVFPeq}
  \pa_t u_i = \sigma\Delta u_i + \diver\sum_{j=1}^n\int_\dom M_{ij}(x,y)
	u_i(x)u_j(y)dxdy, \quad i=1,\ldots,n,
\end{equation}
where we simplified the diffusion part involving $\sigma$.
In the present work, we simplify the problem further by assuming that the interaction
kernels $M_{ij}$ have the special form $M_{ij}(x,y)=\na K_{ij}(x-y)$.
Our main objective is to advance the theory of nonlocal cross-diffusion systems.
This theoretical study can allow us in the future to design and analyze
efficient numerical schemes for the discretization of \eqref{1.eq}--\eqref{1.nonloc}.

Most nonlocal models studied in the literature describe a single species.
A simple example is $\pa_t u = \diver(uv)$ with $v = \na(K*u)$.
This corresponds to the mass continuity equation for the
density $u$ with a nonlocal velocity $v$. An $L^p$ theory for this
equation was provided in \cite{BTR11}, while the Wasserstein gradient-flow structure 
was explored in \cite{CDFLS11}. In the machine learning context, the equation can be 
seen as the mean-field limit of infinitely many hidden network units 
\cite{MMM19,SiSp19}. Besides, in addition to the cited papers, 
there exist some works dealing with the existence of solutions to multispecies 
nonlocal systems of the form \eqref{1.eq}--\eqref{1.nonloc}. Indeed, for two
species and symmetrizable cross-interaction potentials (i.e.\ $K_{12}=\alpha K_{21}$ 
for some $\alpha>0$) without diffusion $\sigma=0$, a complete existence and 
uniqueness theory for measure solutions to \eqref{1.eq}--\eqref{1.nonloc} 
in the whole space with smooth convolution kernels
was established in \cite{DiFa13} using the Wasserstein gradient-flow theory. 
A nonlocal system with size exclusion was analyzed in \cite{BBP17}, using entropy 
methods. In \cite{GaVe19}, a nonlocal version of the Shigesada--Kawasaki--Teramoto 
(SKT) cross-diffusion system, where the diffusion operator is replaced by an integral 
diffusion operator, was analyzed. Finally, closer to our study, in \cite{DiMo21}, 
the authors show the existence of weak solutions to a nonlocal version of the 
SKT system, where the nonlocalities are similar to the ones considered in this paper. 
Assuming some regularity on the convolution kernels, their proof is based on the 
so-called duality method \cite{DLMT15,LM17}. They also proved a localization 
limit result. 

All these works, except \cite{DiMo21}, are concerned with two-species models. 
Compared to previous results, we allow for an arbitrary number of species and 
nondifferentiable kernel functions. In the following section, we explain 
our approach in detail.

\subsection{Description of our approach}

The mathematical difficulties are the 
cross-diffusion terms and the nonlocality, which exclude the application of
standard techniques like maximum principles and regularity theory. 
For instance, it is well known that nonlocal diffusion operators generally do
not possess regularizing effects on the solution \cite{AMRT10}. 
The key of our analysis is the observation that the nonlocal system possesses,
like the associated local one, {\em two} entropies, namely
the Shannon-type entropy $H_1$ \cite{Sha48}
and the Rao-type entropy $H_2$ \cite{Rao82},
\begin{align*}
  H_1(u) &= \sum_{i=1}^n\int_\dom\pi_i u_i(\log u_i-1)dx, \\
  H_2(u) &= \frac12\sum_{i,j=1}^n\int_\dom\int_\dom\pi_i K_{ij}(x-y)u_i(x)u_j(y)dxdy,
\end{align*}
if the kernels are in detailed balance and positive definite
in the sense specified below and for suitable numbers $\pi_i>0$.
A formal computation
that is made rigorous below shows that the following entropy inequalities hold:
\begin{align}
  & \frac{dH_1}{dt} + 4\sigma\sum_{i=1}^n\int_\dom\pi_i |\na\sqrt{u_i}|^2 dx
	= -\sum_{i,j=1}^n\int_\dom\int_\dom\pi_i 
	K_{ij}(x-y)\na u_i(x)\cdot\na u_j(y)dxdy, \label{1.H1} \\
  & \frac{dH_2}{dt} + \sum_{i=1}^n\int_\dom\pi_i u_i|\na p_i[u]|^2 dx
	= -\sigma\sum_{i,j=1}^n\int_\dom\int_\dom\pi_i
	K_{ij}(x-y)\na u_i(x)\cdot\na u_j(y)dxdy. \label{1.H2}
\end{align}
These computations are valid if $K_{ij}$ is in detailed balance, which means that
there exist $\pi_1,\ldots,\pi_n>0$ such that 
\begin{equation}\label{1.db}
  \pi_i K_{ij}(x-y) = \pi_j K_{ji}(y-x)\quad\mbox{for all }i,j=1,\ldots,n,\
	x,y\in\dom.
\end{equation}
We recognize these identities as a generalized detailed-balance condition 
for the Markov chain associated to $(K_{ij}(x-y))$ (for fixed $x-y$), 
and in this case, $(\pi_1,\ldots,\pi_n)$ is the
corresponding reversible measure. The functionals $H_1$ and $H_2$ are Lyapunov
functionals if $(\pi_i K_{ij})$ is positive definite in the sense
\begin{equation}\label{1.pd}
  \sum_{i,j=1}^n\int_\dom\int_\dom\pi_i K_{ij}(x-y)v_i(x)v_j(y)dxdy\ge 0
	\quad\mbox{for all }v_i,v_j\in L^2(\dom).
\end{equation}
This condition generalizes the usual definition
of positive definite kernels to the multi-species case \cite{BPGL04}.
Examples of kernels that satisfy \eqref{1.db} and \eqref{1.pd} are given in 
Remark \ref{example}.
Because of the nonlocality, we cannot conclude $L^2(\dom)$ estimates for
$u_i$ and $\na u_i$ like in the local case; see \cite{JuZu20} and Appendix
\ref{local}. We deduce from
\eqref{1.H1} only bounds for $u_i\log u_i$ in $L^1(\dom)$ and 
$\sqrt{u_i}$ in $H^1(\dom)$.

These bounds are not sufficient to pass to the limit in the approximate problem. 
In particular, we cannot identify the limit of the product $u_i\na p_i[u]$, 
since $u_i$ and $\na p_i[u]$ are elements in larger spaces than $L^2(\dom)$. 
We solve this issue by exploiting the uniform $L^2(\dom)$ bound for 
$\sqrt{u_i}\na p_i[u]$ from \eqref{1.H2} and prove a ``compensated compactness''
lemma (see Lemma \ref{lem.ident} in Appendix \ref{aux}): If $u_\eps\to u$ strongly
in $L^p(\dom)$, $v_\eps\rightharpoonup v$ weakly in $L^p(\dom)$, and
$u_\eps v_\eps\rightharpoonup w$ weakly in $L^p(\dom)$ for some $1<p<2$ then
$uv=w$. 
The estimates from \eqref{1.H1}--\eqref{1.H2} are the key for the proof of the 
global existence of weak solutions to \eqref{1.eq}--\eqref{1.nonloc}. 

As a second result, we prove the weak-strong uniqueness of solutions, i.e., if
$u$ is a weak solution to \eqref{1.eq}--\eqref{1.nonloc}
satisfying $u_i\in L^2(0,T;H^1(\dom))$ and if $v$ is a ``strong'' solution
to this problem with the same initial data, then $u(t)=v(t)$ for a.e.\ $t\ge 0$. 
The proof uses the relative entropy
$$
  H(u|v) = \sum_{i=1}^n\int_\dom\pi_i\big(u_i(\log u_i-1)-u_i\log v_i+v_i\big)dx,
$$
a variant of which was used in \cite{Fis17} for reaction-diffusion systems 
in the context of renormalized solutions and
later extended to Shigesada--Kawasaki--Teramoto systems \cite{ChJu19}.
The recent work \cite{Hop21} generalizes this approach to more general 
Shigesada--Kawasaki--Teramoto as well as energy-reaction-diffusion systems.
Originally, the relative entropy method was devised for conservation laws to estimate
the $L^2$ distance between two solutions \cite{Daf79,DiP79}.
Up to our knowledge, we apply this techniques for the first time to 
nonlocal cross-diffusion systems.
The idea is to differentiate $H(u|v)$ and to derive the inequality
$$
  H(u(t)|v(t)) \le C\sum_{i=1}^n\int_0^t\|u_i-v_i\|_{L^1(\dom)}^2ds
	\quad\mbox{for }t>0.
$$
The Csisz\'ar--Kullback--Pinsker inequality \cite[Theorem A.2]{Jue16}
allows us to estimate the relative entropy from below by
$\|u_i(t)-v_i(t)\|_{L^1(\dom)}^2$, up to some factor. Then 
Gronwall's lemma implies that $u_i(t)=v_i(t)$ for a.e.\ $t>0$. 
The application of this inequality is different 
from the proof in \cite{ChJu19,Fis17,Hop21}, where the relative entropy is estimated
from below by $|u_i-v_i|^2$ on the set $\{u_i\le K\}$.
The difference originates from the nonlocal terms. 
Indeed, if $K_{ij}$ is bounded,
$$
  \sum_{i,j=1}^n\int_\dom\int_\dom\pi_i K_{ij}(x-y)(u_i-v_i)(x)(u_j-v_j)(y)dxdy
	\le C\sum_{i=1}^n\bigg(\int_\dom|u_i-v_i|dx\bigg)^2,
$$
leading to an estimate in $L^1(\dom)$.
In the local case, the associated estimate yields an $L^2(\dom)$ estimate:
$$
  \sum_{i,j=1}^n\int_\dom\pi_i a_{ij}(u_i-v_i)(x)(u_j-v_j)(x)dx
	\le C\sum_{i=1}^n\int_\dom|u_i-v_i|^2 dx.
$$
As the densities $u_i$ may be only nonnegative, we cannot use $\pi_i\log u_i$
as a test function in \eqref{1.eq} to compute $dH(u|v)/dt$. This issue is overcome
by regularizing the entropy by using $\log(u_i+\eps)$ for some $\eps>0$ as a
test funtion and then to pass to the limit $\eps\to 0$.

We observe that the uniqueness of weak solutions to cross-diffusion systems
is a delicate task, and there are only few results in the literature.
Most of the results are based on the fact that the total density $\sum_{i=1}^n u_i$
satisfies a simpler equation for which uniqueness can be shown; see 
\cite{BBP17,ChJu18}. A duality method for a nonlocal version of the 
Shigesada--Kawasaki--Teramoto system was used in \cite{GaVe19}.
In \cite{BBEP20}, a weak-strong uniqueness result on a cross-diffusion system, 
based on $L^2$ estimates, was shown.

The bounds obtained in the proof of our existence result are independent 
of the kernels, such that we can
perform the localization limit, our third main result. For this, we assume that
$K_{ij}=B_{ij}^\eps\to a_{ij}\delta_0$ as $\eps\to 0$ in
the sense of distributions, where $\delta_0$ is the Dirac delta distribution.
Then, if $u^\eps$ is a weak solution to \eqref{1.eq}--\eqref{1.nonloc}, we prove that
$u_i^\eps\to u_i$ strongly in $L^1(\dom\times(0,T))$, and the limit $u$ solves the local
system \eqref{1.eq} and \eqref{1.loc}. As a by-product, we
obtain the global existence of weak solutions to this problem; see Appendix
\ref{local} for the precise statement.

We summarize our main results:
\begin{itemize}
\item global existence of weak solutions to the nonlocal system 
\eqref{1.eq}--\eqref{1.nonloc} for nondifferentiable positive definite kernels 
in detailed balance;
\item weak-strong uniqueness of solutions to the nonlocal system;
\item localization limit to the local system \eqref{1.eq} and \eqref{1.loc}.
\end{itemize}

The paper is organized as follows. Our hypotheses and main results are made
precise in Section \ref{sec.main}. The global existence of weak solutions to
the nonlocal system and some regularity results are 
proved in Section \ref{sec.ex}. 
The weak-strong uniqueness result is shown in Section \ref{sec.wsn}.
In Section \ref{sec.loc}, the localization
limit, based on the a priori estimates of Section \ref{sec.ex}, is performed.
Finally, we collect some auxiliary lemmas in Appendix \ref{aux} and 
state a global existence result for the local system \eqref{1.eq}
and \eqref{1.loc} in Appendix \ref{local}.

%%%%%%%%%%%%%%%%%%%%%%%%%%%%%%%%%%%%%%%%%%%%%%%%%%%%%%%%%%%%%%%%%%%%%%%%%%%%%%%

\section{Main results}\label{sec.main}

We collect the main theorems which are proved in the subsequent sections.
We impose the following hypotheses:

\begin{labeling}{(A44)}
\item[(H1)] Data: Let $d\ge 1$, $T>0$, $\sigma>0$, 
and $u^0\in L^2(\dom)$ satisfies $u_i^0\ge 0$ in $\dom$, $i=1,\ldots,n$.
\item[(H2)] Regularity: $K_{ij}\in L^{s}(\dom)$ for $i,j=1,\ldots,n$, where
$s=d/2$ if $d>2$, $s>1$ if $d=2$, and $s=1$ if $d=1$.
\item[(H3)] Detailed balance: There exist $\pi_1,\ldots,\pi_n>0$ such that 
$\pi_i K_{ij}(x-y)=\pi_j K_{ji}(y-x)$ for all $i,j=1,\ldots,n$
and for a.e.\ $x,y\in\dom$.
\item[(H4)] Positive definiteness: For all $v_1,\ldots,v_n\in L^2(\dom)$, 
it holds that
$$
 \sum_{i,j=1}^n\int_\dom\int_\dom\pi_i K_{ij}(x-y)v_i(x)v_j(y)dxdy\ge 0.
$$
\end{labeling}

%The positive
%lower bound for $u_i^0$ is not essential, we may assume that $u_i^0\ge 0$ in
%$\dom$. This can be shown as in \cite[Sec.~2]{ChJu19a}
%by performing the limit $m_0\to 0$, since
%the estimates depend only on $u_i^0\log u_i^0$ or $(u_i^0)^2$, 
%which are finite for all $u_i^0\ge 0$.
We need the same diffusivity $\sigma$ for all species, since otherwise
we cannot prove that the Rao-type functional $H_2$ is a Lyapunov functional.
The reason is the mixing of the species in the definition of $H_2$.

\begin{remark}[Kernels satisfying Hypotheses (H2)--(H4)]\rm\label{example}
Kernels satisfying Hypothesis (H4) with $n=1$ can be characterized by 
Mercer's theorem \cite{BPGL04,Mer09}. 
A simple example satisfying Hypotheses (H2)--(H4)
is $K_{ij}=a_{ij}\mathrm{1}_{B}$ for
$i,j=1,\ldots,n$, where $\mathrm{1}_{B}$ is the characteristic function of
the ball $B$ around the origin and there exist $\pi_1,\ldots,\pi_n>0$ such that
$(\pi_i a_{ij})\in\R^{n\times n}$ is symmetric and positive definite.

Another example that satisfies Hypotheses (H2)--(H4)
is given by the Gaussian kernel $B(|x-y|)=(2\pi)^{-d/2}\exp(-|x-y|^2/2)$. We define
for $i,j=1,\ldots,n$ and $x,y\in\R^d$,
\begin{equation*}
  K_{ij}(x-y) = B_{ij}^\eps(x-y)
	:= \frac{a_{ij}}{(2\pi\eps^2)^{d/2}}\exp\bigg(-\frac{|x-y|^2}{2\eps^2}\bigg),
\end{equation*}
where $\eps>0$ and $a_{ij}\ge 0$ are such that the matrix $(\pi_i a_{ij})$ is
symmetric and positive definite for some $\pi_i>0$. Thus, Hypothesis (H3)
holds. Hypothesis (H4) can be verified as follows. The identity
$$
  \frac{e^{-|x-y|^2/(2\eps^2)}}{(2\pi\eps^2)^{d/2}} 
	= \int_{\R^d}\frac{e^{-|x-z|^2/\eps^2}}{(\pi\eps^2)^{d/2}}
	\frac{e^{-|y-z|^2/\eps^2}}{(\pi\eps^2)^{d/2}}dz,
$$
shows that
\begin{align*}
  &\sum_{i,j=1}^n\int_\dom\int_\dom\pi_i K_{ij}(x-y)v_i(x)v_j(y)dxdy
	= \sum_{i,j=1}^n \int_\dom\int_\dom\pi_i a_{ij}
	\frac{e^{-|x-y|^2/(2\eps^2)}}{(2\pi\eps^2)^{d/2}}v_i(x)v_j(y)dxdy \\
	&= \sum_{i,j=1}^n \pi_i a_{ij}\int_{\R^d}\int_\dom
	\bigg(\frac{e^{-|x-z|^2/\eps^2}}{(\pi\eps^2)^{d/2}}v_i(x)\bigg)dx
	\int_\dom\bigg(\frac{e^{-|y-z|^2/\eps^2}}{(\pi\eps^2)^{d/2}}v_j(y)\bigg)dydz \\
	&\ge \frac{\alpha}{(\pi\eps^2)^{d}}\sum_{i=1}^n\int_{\R^d}
	\bigg(\int_\dom e^{-|x-z|^2/\eps^2}v_i(x)dx\bigg)^2 dz \ge 0,
\end{align*}
where $\alpha>0$ is the smallest eigenvalue of $(\pi_i a_{ij})$.
This proves the positive definiteness of $K_{ij}$.
Note that $B_{ij}^\eps\to a_{ij}\delta_0$ as $\eps\to 0$ in the sense of
distributions.

We can construct further examples from the Gaussian kernel. For instance,
$$
  K_{ij}(x-y) = \frac{a_{ij}}{1+|x-y|^2}, \quad i,j=1,\ldots,n,\ x,y\in\R^d,
$$
satisfies Hypothesis (H4), since
$$
  \frac{1}{1+|x-y|^2} = \int_0^\infty e^{-s(1+|x-y|^2)}ds,
$$
and $B(x-y)=\exp(-s(1+|x-y|^2))$ is positive definite.
\qed\end{remark}

We call $u=(u_1,\ldots,u_n)$ a {\em weak solution} to 
\eqref{1.eq}--\eqref{1.nonloc} if for all test functions $\phi_i\in
L^{d+2}(0,T;$ $W^{1,d+2}(\dom))$, $i=1,\ldots,n$, it holds that
\begin{equation}\label{ex.wf}
  \int_0^T\langle\pa_t u_i,\phi_i\rangle dt
	+ \sigma\int_0^T\int_{\dom}\na u_i\cdot\na\phi_i dxdt
	= -\int_0^T\int_{\dom}u_i\na p_i[u]\cdot\na\phi_i dxdt,
\end{equation}
where $\langle\cdot,\cdot\rangle$ denotes the dual pairing between
$W^{1,d+2}(\dom)'$ and $W^{1,d+2}(\dom)$,
and the initial datum $u_i(0)=u_i^0$ is satisfied in the sense of 
$W^{1,d+2}(\dom)'$.

First, we show the global existence of weak solutions. Let $Q_T=\dom\times(0,T)$.

\begin{theorem}[Global existence]\label{thm.ex}
Let Hypotheses (H1)--(H4) hold. 
Then there exists a global weak solution $u=(u_1,\ldots,u_n)$ 
to \eqref{1.eq}--\eqref{1.nonloc} satisfying 
$u_i\ge 0$ in $Q_T$ and
\begin{equation}\label{1.regul}
\begin{aligned}
  & u_i^{1/2}\in L^2(0,T;H^1(\dom)),\quad u_i\in L^{1+2/d}(Q_T)\cap
	L^q(0,T;W^{1,q}(\dom)), \\
	& \pa_t u_i\in L^q(0,T;W^{-1,q}(\dom)), \quad
  u_i\na p_i[u]\in L^q(Q_T),
\end{aligned}
\end{equation}
where $q=(d+2)/(d+1)$ and $i=1,\ldots,n$. The initial datum in \eqref{1.eq}
is satisfied in the sense of $W^{-1,q}(\dom):=W^{1,d+2}(\dom)'$.
Moreover, the following entropy inequalities hold:
\begin{align}
  H_1(u(t)) + 4\sigma\sum_{i=1}^n \int_0^t\int_\dom\pi_i|\na u_i^{1/2}|^2 dxds 
	&\le H_1(u^0), \label{ex.H1} \\
  H_2(u(t)) + \sum_{i=1}^n \int_0^t\int_\dom\pi_i u_i|\na p_i[u]|^2 dxds 
	&\le H_2(u^0). \label{ex.H2}
\end{align}
\end{theorem}

Unfortunately, we have been not able to allow for vanishing diffusion
$\sigma=0$, since estimate \eqref{1.H2} is not sufficient and assumption
\eqref{1.pd} is too weak to conclude gradient bounds. On the other hand,
we can allow for arbitrarily small $\sigma>0$, which means that cross-diffusion
may dominate diffusion.

Imposing more regularity on the kernel functions, we can derive $H^1(\dom)$
regularity for $u_i$, which is needed for the weak-strong uniqueness result.

\begin{proposition}[Regularity]\label{prop.regul}\sloppy
Let Hypotheses (H1)--(H4) hold and let $\na K_{ij}\in$ $L^{d+2}(\dom)$
for $i,j=1,\ldots,n$. Then there exists a weak solution $u=(u_1,\ldots,u_n)$ to 
\eqref{1.eq}--\eqref{1.nonloc} satisfying $u_i\ge 0$ in $\dom$ and
$$
  u_i\in L^2(0,T;H^1(\dom)), \quad \pa_t u_i\in L^2(0,T;H^{-1}(\dom)), \quad
	\na p_i[u]\in L^\infty(0,T;L^\infty(\dom)).
$$
Moreover, if additionally $\na K_{ij}$, $\Delta K_{ij}\in L^\infty(\dom)$
and $m_0\le u_i^0\le M_0$ in $\dom$ then $0<m_0 e^{-\lambda t}\le u_i(t)
\le M_0 e^{\lambda t}$ in $\dom$ for $t>0$, where $\lambda>0$ depends on
$\Delta K_{ij}$ and $u^0$.
\end{proposition}

The proof of the $H^1(\dom)$ regularity
is based on standard $L^2$ estimates if $\na K_{ij}\in L^\infty(\dom)$.
The difficulty is the reduced regularity $\na K_{ij}\in L^{d+2}(\dom)$, which requires
some care. Indeed, the test function $u_i$ in the weak formulation of
\eqref{1.eq} leads to a cubic term, which is reduced to a subquadratic term 
for $\na u_i$ by combining the Gagliardo--Nirenberg inequality 
and the uniform $L^1(\dom)$ bound for $u_i$.

Similar lower and upper bounds as in Proposition \ref{prop.regul}
were obtained in \cite{DiMo21} with a different proof. Since the $L^\infty$
bounds depend on the derivatives of $K_{ij}$,
they do not carry over in the localization limit to the local system.
In fact, it is an open problem whether the local system \eqref{1.eq}
and \eqref{1.loc} possesses {\em bounded} weak solutions.
The proposition also holds for kernel functions $K_{ij}(x,y)$ that are
used in neural network theory; see Remark \ref{rem.kernel}.

\begin{theorem}[Weak-strong uniqueness]\label{thm.ws}
Let $K_{ij}\in L^\infty(\R^d)$ for $i,j=1,\ldots,n$.
Let $u$ be a nonnegative weak solution to \eqref{1.eq}--\eqref{1.nonloc} satisfying
\eqref{1.regul} as well as $u_i\in L^2(0,T;H^1(\dom))\cap H^1(0,T;H^{-1}(\dom))$,
and let $v=(v_1,\ldots,v_n)$ be a ``strong'' solution to 
\eqref{1.eq}--\eqref{1.nonloc}, i.e.\ a weak solution to 
\eqref{1.eq}--\eqref{1.nonloc} satisfying
$$
  c\le v_i\le C\quad\mbox{in }Q_T,
	\quad \pa_t v_i\in L^2(0,T;H^{-1}(\dom)), 
	\quad v_i\in L^\infty(0,T;W^{1,\infty}(\dom)),
$$
for some $C\ge c>0$ and having the same initial data as $u$. 
Then $u(x,t)=v(x,t)$ for a.e.\ $(x,t)\in\dom\times(0,T)$. 
\end{theorem}

The existence of a strong solution $v_i$ to \eqref{1.eq}--\eqref{1.nonloc}
was proved in \cite[Prop.~1]{CDJ19}, but only locally in time and in the whole
space setting. While the proof can be adapted to the case of a torus, 
it is less clear how to extend it globally in time.
Theorem \ref{thm.ws} cannot be extended in a straightforward way to the whole-space
case since $v_i\ge c>0$ would be nonintegrable. In the case of the
Maxwell--Stefan cross-diffusion system on a bounded domain $\Omega\subset\R^d$, 
it is possible to relax the lower bound to $v_i>0$ a.e.\ and 
$\log v_i\in L^2(0,T;H^1(\Omega))$ \cite{HJT21}. The proof could be possibly
extended to the whole space, but the computations in \cite{HJT21}
are made rigorous by exploiting
the specific structure of the Maxwell--Stefan diffusion coefficients.

For the localization limit, we choose the radial kernel
\begin{equation}\label{2.K}
  K_{ij}^\eps(x-y) = \frac{a_{ij}}{\eps^{d}}B\bigg(\frac{|x-y|}{\eps}\bigg),
	\quad i,j=1,\ldots,n,\ x,y\in\dom,
\end{equation}
where $B\in C^0(\R)$, $\operatorname{supp}(B)\subset (-1,1)$,
$\int_{-1}^1 B(|z|)dz=1$, and $a_{ij}\ge 0$ is such that
$(\pi_i a_{ij})$ is symmetric and positive definite for some $\pi_i>0$, $i=1,\ldots,n$.

\begin{theorem}[Localization limit]\label{thm.loc}
Let $K_{ij}^\eps$ be given by \eqref{2.K} and satisfying Hypothesis (H4). 
Let $u^\eps$ be 
the weak solution to \eqref{1.eq}--\eqref{1.nonloc},
constructed in Theorem \ref{thm.ex}. Then there
exists a subsequence of $(u^\eps)$ that is not relabeled such that, as $\eps\to 0$,
$$
  u_\eps\to u\quad\mbox{strongly in }L^2(0,T;L^{d/(d-1)}(\dom)),
$$
if $d \geq 2$ and strongly in $L^2(0,T;L^r(\dom))$ for any $r<\infty$ if $d=1$. 
Moreover, $u$ is a nonnegative weak solution to \eqref{1.eq} and \eqref{1.loc}.
\end{theorem}

The existence of global weak solutions to \eqref{1.eq} and 
\eqref{1.loc} can also be proved for any bounded domain; see Appendix \ref{local}.

%%%%%%%%%%%%%%%%%%%%%%%%%%%%%%%%%%%%%%%%%%%%%%%%%%%%%%%%%%%%%%%%%%%%%%%%%%%%%%%

\section{Global existence for the nonlocal system}\label{sec.ex}

We prove the global existence of a nonnegative weak solution $u$ to
\eqref{1.eq}--\eqref{1.nonloc} and show the regularity properties of
Proposition \ref{prop.regul}. Since the proof is based on
the entropy method similar to \cite[Chapter 4]{Jue16}, we sketch the standard
arguments and focus on the derivation of uniform estimates.

{\em Step 1: Solution of an approximated system.}
Let $T>0$, $N\in\N$, $\tau=T/N$, $\delta>0$, and $m\in\N$ with $m>d/2+1$. 
We proceed by induction over $k\in\N$.
Let $u^{k-1}\in L^2(\dom;\R^n)$ be given. (The superindex $k$ refers to
the time step $t_k=k\tau$.)
Set $u_i(w)=\exp(w_i/\pi_i)>0$.
We wish to find $w^k\in H^m(\dom;\R^n)$ to the approximated system
\begin{align}
  \frac{1}{\tau}&\int_\dom(u(w^k)-u^{k-1})\cdot\phi dx
	+ \sigma\sum_{i=1}^n\int_\dom\na u_i(w^k)\cdot\na\phi_i dx 
	+ \delta b(w^k,\phi) \nonumber \\ 
	&= -\sum_{i=1}^n\int_\dom u_i(w^k)\na p_i[u(w^k)]\cdot\na\phi_i dx, \label{ex.approx}
\end{align}
for $\phi=(\phi_1,\ldots,\phi_n)\in H^m(\dom;\R^n)$. The bilinear form
$$
  b(w^k,\phi) = \int_\dom\bigg(\sum_{|\alpha|=m}D^\alpha w^k\cdot D^\alpha\phi
	+ w\cdot\phi\bigg)dx,
$$
is coercive on $H^m(\dom;\R^n)$, i.e.\ $b(w^k,w^k) \ge C\|w^k\|_{H^m(\dom)}^2$
for some $C>0$,
as a consequence of the generalized Poincar\'e--Wirtinger inequality.
By the fixed-point argument on the space $L^\infty(\dom;\R^n)$ 
used in \cite[Section 4.4]{Jue16}, it is
sufficient to derive a uniform bound for $w_i^k$ in $H^m(\dom)$, 
which embeddes compactly into $L^\infty(\dom)$. To this end, we use
the test function $\phi_i=w_i^k=\pi_i\log u_i^k$ (with $u_i^k:=u_i(w^k)$)
in \eqref{ex.approx}:
\begin{align*}
  \sum_{i=1}^n&\frac{\pi_i}{\tau}\int_\dom(u_i^k-u_i^{k-1})\cdot\log u_i^k dx
	+ 4\sigma\sum_{i=1}^n\pi_i\int_\dom|\na(u_i^k)^{1/2}|^2 dx
  + \delta b(w^k,w^k) \\
	&= -\sum_{i=1}^n\int_\dom u_i^k\na p_i[u^k]\cdot\na w_i^k dx
	= -\sum_{i=1}^n\int_\dom\pi_i\na p_i[u^k]\cdot\na u_i^k dx,
\end{align*}
where we used the identity $u_i^k\na w_i^k=\pi_i \na u_i^k$.
An integration by parts gives
\begin{equation}\label{ex.ibp}
  \int_\dom\na K_{ij}(x-y)u_j^k(y)dy
	= \int_\dom\na K_{ij}(z)u_j^k(x-z)dz
	= \int_\dom K_{ij}(x-y)\na u_j^k(y)dy.
\end{equation}
Thus, in view of definition \eqref{1.nonloc} of $p_i[u^k]$
and Hypothesis (H4),
\begin{align*}
  \sum_{i=1}^n&\frac{\pi_i}{\tau}\int_\dom(u_i^k-u_i^{k-1})\cdot\log u_i^k dx
	+ 4\sigma\sum_{i=1}^n\pi_i\int_\dom|\na(u_i^k)^{1/2}|^2 dx
  + \delta b(w^k,w^k) \\
	&= -\sum_{i,j=1}^n\int_\dom\int_\dom\pi_i K_{ij}(x-y)\na u_j^k(y)
	\cdot\na u_i^k(x) dxdy \le 0.
\end{align*}
The convexity of $f(z)=z(\log z-1)$ for $z\ge 0$ implies that
$f(z)-f(y)\le f'(z)(z-y)$ for $y,z>0$. Using this inequality to estimate the
first integral and the coercivity of $b(w^k,w^k)$ to estimate the third term, 
we find that
\begin{align}
  \frac{1}{\tau}\sum_{i=1}^n &\int_\dom\pi_i\big(u_i^k(\log u_i^k-1) 
	- u_i^{k-1}(\log u_i^{k-1}-1)\big)dx \nonumber \\
	&{}+ 4\sigma \sum_{i=1}^n\pi_i\|\na(u^k_i)^{1/2}\|_{L^2(\dom)}^2 
	+ \delta C\sum_{i=1}^n\|w_i^k\|_{H^m(\dom)}^2 \le 0. \label{ex.ei0}
\end{align}
This provides a uniform estimate for $w^k$ in $H^m(\dom)
\hookrightarrow L^\infty(\dom)$ (not uniform in $\delta$), 
necessary to conclude the fixed-point
argument and giving the existence of a solution $w^k\in H^m(\dom;\R^n)$
to \eqref{ex.approx}. This defines $u^k:=u(w^k)$, finishing the induction step.

To derive further uniform estimates, we wish to use $\psi_i=\pi_i p_i[u^k]$ as a 
test function in \eqref{ex.approx}. However, we cannot estimate the term
$\delta b(w^k,\psi)$ appropriately. Therefore, we perform the limits
$\delta\to 0$ and $\tau\to 0$ separately.

{\em Step 2: Limit $\delta\to 0$.}
Let $u^\delta=(u_1^\delta,\ldots,u_n^\delta)$ 
with $u_i^\delta=u_i(w^k)$ be a solution to \eqref{ex.approx} and let
$w_i^\delta=\pi_i\log u_i^\delta$ for $i=1,\ldots,n$ (slightly abusing the
notation). Estimate \eqref{ex.ei0}
and the Poincar\'e--Wirtinger inequality show that $(u_i^\delta)^{1/2}$ is 
uniformly bounded in $H^1(\dom)$ and, by Sobolev's embedding, in
$L^{r_1}(\dom)$, where $r_1=2d/(d-2)$ if $d>2$, $r_1<\infty$ if $d=2$, and 
$r_1=\infty$ if $d=1$. Therefore, $\na u_i^\delta 
= 2(u_i^\delta)^{1/2}\na(u_i^\delta)^{1/2}$ is uniformly bounded in 
$L^{r_2}(\dom)$, where $r_2=d/(d-1)$ if $d>2$, $r_2<2$ if $d=2$, and $r_2=2$ if $d=1$. 
By Sobolev's embedding, $(u_i^\delta)$ is relatively compact
in $L^{r}(\dom)$ for $r<r_1/2$, and
there exists a subsequence that is not relabeled such that, as $\delta\to 0$,
\begin{align*}
  u_i^\delta\to u_i &\quad\mbox{strongly in }L^{r}(\dom), \quad r<r_1/2, \\
  \na u_i^\delta\rightharpoonup \na u_i &\quad\mbox{weakly in }L^{r_2}(\dom), \\
	\delta w_i^\delta\to 0 &\quad\mbox{strongly in }H^m(\dom).
\end{align*}
In particular, we have, up to a subsequence, $u_i^\delta\to u_i$ a.e.\ and
$(u_i^\delta)$ is dominated by some function in $L^r(\dom)$. 
By dominated convergence,
$p_i[u^\delta]\to p_i[u]$ a.e. Young's convolution inequality (see
Lemma \ref{lem.young} in Appendix \ref{aux}) shows that, 
for $d>2$,
\begin{align*}
  \|p_i[u^\delta]\|_{L^\infty(\dom)} 
	&\le \sum_{j=1}^n\bigg\|\int_\dom K_{ij}(\cdot-y)u_j^\delta(y)dy
	\bigg\|_{L^\infty(\dom)} \\
  &\le \sum_{j=1}^n\|K_{ij}\|_{L^{d/2}(\dom)}\|u_j^\delta\|_{L^{d/(d-2)}(\dom)}\le C,
\end{align*}
In a similar way, we can prove that $(p_i[u^\delta])$ is bounded in
$L^r(\dom)$ for any $r<\infty$ if $d=2$ and in $L^\infty(\dom)$ if $d=1$, assuming
that $K_{ij}\in L^1(\dom)$. Lemma \ref{lem.conv} in Appendix \ref{aux} 
implies that $p_i[u^\delta]\to p_i[u]$ strongly in
$L^r(\dom)$ for any $r<\infty$. Furthermore, if $d>2$,
$$
  \|\na p_i[u^\delta]\|_{L^{r_3}(\dom)} 
	\le \sum_{j=1}^n\|K_{ij}\|_{L^{d/2}(\dom)}\|\na u_j^\delta\|_{L^{d/(d-1)}(\dom)}
	\le C, 
$$
where $r_3=d$, and $(\na p_i[u^\delta])$ is bounded in $L^{r_3}(\dom)$ for 
some $r_3>2$ if $d=2$ and for $r_3=2$ if $d=1$. Hence, for a subsequence,
$$
  \na p_i[u^\delta]\rightharpoonup \na p_i[u]\quad\mbox{weakly in }L^{r_3}(\dom).
$$
It follows that $(u_i^\delta\na p_i[u^\delta])$ is bounded in
$L^{\min\{2,d/(d-1)\}}(\dom)$ and
$$
  u_i^\delta\na p_i[u^\delta]\rightharpoonup u_i\na p_i[u]\quad\mbox{weakly in }
	L^{\min\{2,d/(d-1)\}}(\dom).
$$
Thus, we can pass to the limit $\delta\to 0$ in \eqref{ex.approx} to conclude that
$u_i^k:=u_i\ge 0$ for $i=1,\ldots,n$ solves
\begin{equation}\label{ex.approx2}
  \frac{1}{\tau}\int_\dom(u^k-u^{k-1})\cdot\phi dx 
	+ \sigma\sum_{i=1}^n\int_\dom\na u_i^k\cdot\na\phi_i dx
	= -\sum_{i=1}^n\int_\dom u_i^k\na p_i[u^k]\cdot\na\phi_i dx,
\end{equation}
for all test functions $\phi_i\in W^{1,r_3}(\dom)$. 
Observe that $p_i[u^k]$ is an element of the space $W^{1,r_3}(\dom)$
and is an admissible test function; this will be used in the next step.

{\em Step 3: Uniform estimates.}
We introduce the piecewise constant in time functions $u^{(\tau)}(x,t)=u^k(x)$ 
for $x\in\dom$ and $t\in((k-1)\tau,k\tau]$, 
$k=1,\ldots,N$. At time $t=0$, we set 
$u_i^{(\tau)}(\cdot,0)=u_i^0$. Furthermore, we use the time-shift operator 
$(S_\tau u^{(\tau)})(x,t)=u^{k-1}(x)$ for $x\in\dom$, 
$t\in((k-1)\tau,k\tau]$.
Then, summing \eqref{ex.approx2} over $k$, we obtain
\begin{align*}
  \frac{1}{\tau}\int_0^T\int_\dom & (u^{(\tau)}-S_\tau u^{(\tau)})\cdot\phi dxdt
	+ \sigma\sum_{i=1}^n\int_0^T\int_\dom\na u_i^{(\tau)}\cdot\na\phi_i dxdt \\
	&= -\sum_{i=1}^n\int_0^T\int_\dom u^{(\tau)}_i\na p_i[u^{(\tau)}]\cdot\na\phi_i dxdt,
\end{align*}
for piecewise constant functions $\phi:(0,T)\to W^{1,r_3}(\dom;\R^n)$ 
and, by density, for all functions $\phi\in L^2(0,T;W^{1,r_3}(\dom;\R^n))$. 
Summing the entropy inequality \eqref{ex.ei0} from $k=1,\ldots,N$, it follows that
\begin{equation}\label{ex.ulog}
  H_1(u^{(\tau)}(T))
	+ 4\sigma\sum_{i=1}^n\int_0^T\pi_i\|\na(u_i^{(\tau)})^{1/2}\|_{L^2(\dom)}^2 dt
	\le H_1(u^0).
\end{equation}

These bounds allow us to derive further estimates. Since the $L^1\log L^1$
bound dominates the $L^1(\dom)$ norm, we deduce from the Poincar\'e--Wirtinger
inequality that
$$
  \|u_i^{(\tau)}\log u_i^{(\tau)}\|_{L^\infty(0,T;L^1(\dom))}
	+ \|(u_i^{(\tau)})^{1/2}\|_{L^2(0,T;H^1(\dom))} \le C(u^0), \quad i=1,\ldots,n.
$$
This implies, by the Gagliardo--Nirenberg inequality with $\theta=d/(d+2)$, that
\begin{align}
  \|u_i^{(\tau)}\|_{L^{1+2/d}(Q_T)}^{1+2/d}
	&= \int_0^T\|(u_i^{(\tau)})^{1/2}\|_{L^{2+4/d}(\dom)}^{2+4/d}dt \nonumber \\
	&\le C\int_0^T\|(u_i^{(\tau)})^{1/2}\|_{H^1(\dom)}^{\theta(2d+4)/d}
	\|(u_i^{(\tau)})^{1/2}\|_{L^2(\dom)}^{(1-\theta)(2d+4)/d}dt \nonumber \\
	&\le C\|u_i^{(\tau)}\|_{L^\infty(0,T;L^1(\dom))}^{2/d}
	\int_0^T\|(u_i^{(\tau)})^{1/2}\|_{H^1(\dom)}^2 dt \le C(u^0), \label{ex.gn}
\end{align}
and by H\"older's inequality with $q=(d+2)/(d+1)$, 
\begin{align}
  \|\na u_i^{(\tau)}\|_{L^q(Q_T)} 
	&= 2\|(u_i^{(\tau)})^{1/2}\na(u_i^{(\tau)})^{1/2}\|_{L^q(Q_T)} \nonumber \\
	&\le 2\|(u_i^{(\tau)})^{1/2}\|_{L^{2+4/d}(Q_T)}
	\|\na(u_i^{(\tau)})^{1/2}\|_{L^2(Q_T)} \le C. \label{ex.nablau}
\end{align}

These bounds are not sufficient to pass to the limit $\tau\to 0$.
To derive further estimates, we use the test function 
$\phi_i=\pi_i p_i[u^k]\in W^{1,r_3}(\dom)$ in \eqref{ex.approx2}:
\begin{align}
  \frac{1}{\tau}&\sum_{i,j=1}^n\int_\dom\int_\dom\pi_i (u^k_i(x)-u_i^{k-1}(x))
	K_{ij}(x-y)u_j^k(y)dxdy \nonumber \\
	&{}+ \sigma\sum_{i,j=1}^n\int_\dom\int_\dom\pi_i K_{ij}(x-y)
	\na u_i^k(x)\cdot\na u_j^k(y) dxdy
  = -\sum_{i=1}^n\int_\dom\pi_i u_i^k|\na p_i[u^k]|^2 dx. \label{ex.aux1}
\end{align}
Because of the positive definiteness of $\pi_i K_{ij}$, the second term 
on the left-hand side is nonnegative.
Exploiting the symmetry and positive definiteness of $\pi_i K_{ij}$ 
(Hypotheses (H3)--(H4)), the first integral can be estimated from below as
\begin{align*}
  \frac{1}{2\tau}&\sum_{i,j=1}^n\int_\dom\int_\dom\pi_i K_{ij}(x-y)
	\Big(u_i^k(x)u_j^k(y) - u_i^{k-1}(x)u_j^{k-1}(y) \\
	&\phantom{xx}{}+ (u_i^k(x)-u_i^{k-1}(x))
	(u_j^k(y)-u_j^{k-1}(y))\Big)dxdy \\
	&\ge \frac{1}{2\tau}\sum_{i,j=1}^n\int_\dom\int_\dom\pi_i K_{ij}(x-y)
	\big(u_i^k(x)u_j^k(y) - u_i^{k-1}(x)u_j^{k-1}(y)\big)dxdy \\
	&= \frac{1}{\tau}(H_2(u^k) - H_2(u^{k-1})).
\end{align*}
Therefore, we infer from \eqref{ex.aux1} that
$$
  H_2(u^k) + \tau\sum_{i=1}^n\int_\dom\pi_i u_i^k|\na p_i[u^k]|^2 dx
	\le H_2(u^{k-1}),
$$
and summing this inequality from $k=1,\ldots,N$, we have
\begin{equation}\label{ex.up}
  H_2(u^{(\tau)}(T)) 
	+ \sum_{i=1}^n\pi_i\int_0^T\big\|(u_i^{(\tau)})^{1/2}\na p_i[u^{(\tau)}]
	\big\|_{L^2(\dom)}^2 dt \le H_2(u^0).
\end{equation}

The previous bound allows us to derive an estimate for the discrete time derivative.
Indeed, estimates \eqref{ex.gn} and \eqref{ex.up} imply that
$$
  u_i^{(\tau)}\na p_i[u^{(\tau)}] = (u_i^{(\tau)})^{1/2}\cdot
	(u_i^{(\tau)})^{1/2}\na p_i[u^{(\tau)}],
$$
is uniformly bounded in $L^q(Q_T)$, where $q=(d+2)/(d+1)$.
Let $\phi\in L^{q'}(0,T;W^{2,(d+2)/2}(\dom))$, where $q'=d+2$. Then $1/q+1/q'=1$
and
\begin{align*}
  \frac{1}{\tau}&\bigg|\int_0^T\int_\dom(u^{(\tau)}-S_\tau u^{(\tau)})
	\cdot\phi dxdt\bigg| \\
	&\le \sigma\sum_{i=1}^n\|u_i^{(\tau)}\|_{L^{1+2/d}(Q_T)}
	\|\Delta\phi_i\|_{L^{(d+2)/2}(Q_T)} 
	+ \sum_{i=1}^n\|u_i^{(\tau)}\na p_i[u^{(\tau)}]\|_{L^q(Q_T)}
	\|\na\phi_i\|_{L^{q'}(Q_T)} \\
	&\le C\|\phi\|_{L^{q'}(0,T;W^{2,(d+2)/2}(\dom))}.
\end{align*}
We conclude that
\begin{equation}\label{ex.time}
  \tau^{-1}\|u^{(\tau)}-S_\tau u^{(\tau)}\|_{L^q(0,T;W^{2,(d+2)/2}(\dom)')}
	\le C.
\end{equation}

{\em Step 4: Limit $\tau\to 0$.}
Estimates \eqref{ex.ulog} and \eqref{ex.time} allow us to apply the 
Aubin--Lions compactness lemma in the version of \cite{BCJ20} to conclude the
existence of a subsequence that is not relabeled such that, as $\tau\to 0$,
$$
  u_i^{(\tau)} \to u_i \quad\mbox{strongly in }L^2(0,T;L^{d/(d-1)}(\dom)),\ 
	i=1,\ldots,n,
$$
if $d\ge 2$ and strongly in $L^2(0,T;L^r(\dom))$ for any $r<\infty$ if $d=1$. 
Strictly speaking, the version of \cite{BCJ20} holds for the continuous time
derivative, but the technique of \cite{DrJu12} shows that the conclusion of
\cite{BCJ20} also holds for the discrete time derivative. 
Then, maybe for another subsequence, $u_i^{(\tau)}\to u_i$ a.e.\ in $Q_T$, and
we deduce from \eqref{ex.gn} that $u_i^{(\tau)}\to u_i$
strongly in $L^r(Q_T)$ for all $r<1+2/d$ (see Lemma \ref{lem.conv} in Appendix
\ref{aux}). Furthermore, we obtain
from \eqref{ex.ulog}, \eqref{ex.nablau}, \eqref{ex.up}, and \eqref{ex.time} 
the convergences
\begin{align*}
  \na u_i^{(\tau)}\rightharpoonup \na u_i
	&\quad\mbox{weakly in }L^q(Q_T),\ i=1,\ldots,n, \\
  \tau^{-1}(u^{(\tau)}-S_\tau u^{(\tau)})\rightharpoonup \pa_t u_i
	&\quad\mbox{weakly in }L^q(0,T;W^{2,(d+2)/2}(\dom)'), \\
	(u_i^{(\tau)})^{1/2}\na p_i[u^{(\tau)}] \rightharpoonup z_i
	&\quad\mbox{weakly in }L^2(Q_T),
\end{align*}
where $z_i\in L^2(Q_T)$ and $q=(d+2)/(d+1)$. Since $u_i^{(\tau)}\ge 0$, we infer 
that $u_i\ge 0$ in $Q_T$.

{\em Step 5: Identification of the limit.}
We need to identify $z_i$ with $u_i^{1/2}\na p_i[u]$. We show first that 
\begin{equation*}%\label{ex.nap}
  \na p_i[u^{(\tau)}]\rightharpoonup \na p_i[u] \quad\mbox{weakly in }L^q(Q_T).
\end{equation*}
Indeed, it follows from the strong convergence of $(u_i^{(\tau)})$ that 
(up to a subsequence) $K_{ij}(x-y)$ $\times u_j^{(\tau)}(y,t)\to K_{ij}(x-y)u_j(y,t)$ 
for a.e.\ $(y,t)\in Q_T$ and for a.e.\ $x\in\dom$. Hence, because of the 
uniform bounds, $p_i[u^{(\tau)}]\to p_i[u]$ a.e.\ in $Q_T$.
We deduce from Young's convolution inequality and 
the uniform bound for $\na u_i^{(\tau)}$ in $L^q(Q_T)$ that 
$\na p_i[u^{(\tau)}]$ is uniformly bounded in $L^q(Q_T)$. 
(Here, we only need $K_{ij}\in L^1(\dom)$. The estimate for $\na p_i[u^{(\tau)}]$
is better under Hypothesis (H2), but the time regularity cannot be improved.)
Therefore,
$$
  \na p_i[u^{(\tau)}]\rightharpoonup\na p_i[u]\quad\mbox{weakly in }L^q(Q_T).
$$

When $d=2$, we have the convergences $\na p_i[u^{(\tau)}]\rightharpoonup\na p_i[u]$
weakly in $L^{4/3}(Q_T)$ and $(u_i^{(\tau)})^{1/2}\to u_i^{1/2}$ strongly in 
$L^4(Q_T)$, which is sufficient to pass to the limit in the product and to identify
it with $z_i$. 
However, this argument fails when $d>2$, and we need a more sophisticated proof.
The div-curl lemma does not apply, since the exponents of the Lebesgue 
spaces, in which the convergences of $(u_i^{(\tau)})^{1/2}$ and $\na p_i[u^{(\tau)}]$
take place, are not conjugate for $d>2$. 
Also the generalization \cite[Theorem 2.1]{BrDi16}
to nonconjugate exponents cannot be used for general $d$.

Our idea is to exploit the fact that the product converges in a space better
than $L^1$. Then Lemma \ref{lem.ident} in Appendix \ref{aux} immediately implies that
$$
  (u_i^{(\tau)})^{1/2}\na p_i[u^{(\tau)}]\rightharpoonup u_i^{1/2}\na p_i[u]
	\quad\mbox{weakly in }L^q(Q_T).
$$
In fact, estimate \eqref{ex.up} implies that this convergence holds in $L^2(Q_T)$.
Then the strong convergence of $(u_i^{(\tau)})^{1/2}$ in $L^{2}(Q_T)$ gives
$$
  u_i^{(\tau)}\na p_i[u^{(\tau)}]\rightharpoonup u_i\na p_i[u]
	\quad\mbox{weakly in }L^1(Q_T).
$$
In view of the uniform bounds
for $(u_i^{(\tau)})^{1/2}$ in $L^{2+4/d}(Q_T)$ and of $(u_i^{(\tau)})^{1/2}
\na p_i[u^{(\tau)}]$ in $L^2(Q_T)$, the product $u_i^{(\tau)}\na p_i[u^{(\tau)}]$
is uniformly bounded in $L^q(Q_T)$. Thus, the previous weak convergence also
holds in $L^q(Q_T)$. 

{\em Step 6: End of the proof.}
The convergences of the previous step allow us to pass to the limit 
$\tau\to 0$ in \eqref{ex.approx2}, yielding
$$
  \int_0^T\langle \pa_t u_i,\phi_i\rangle dt
	+ \sigma\int_0^T\int_\dom\na u_i\cdot\na\phi_i dxdt
	= -\int_0^T\int_\dom u_i\na p_i[u]\cdot\na\phi_i dxdt,
$$
for all smooth test functions. Because of $\na u_i$, $u_i\na p_i[u]\in L^q(Q_T)$, 
a density argument
shows that the weak formulation holds for all $\phi\in L^{q'}(0,T;W^{1,q'}(\dom))$,
recalling that $q'=d+2$. Then $\pa_t u_i$ lies in the space
$L^q(0,T;W^{-1,q}(\dom))$, where $W^{-1,q}(\dom):=W^{1,q'}(\dom)'$.
The proof that $u(\cdot,0)$ satisfies the initial datum
can be done exactly as in \cite[p.~1980]{Jue15}. Finally, using the convexity
of $H_1$ and the lower semi-continuity of convex functions, the entropy inequalities
\eqref{ex.ulog} and \eqref{ex.up} become \eqref{ex.H1} and \eqref{ex.H2}, 
respectively, in the limit $\tau\to 0$.
This ends the proof of Theorem \ref{thm.ex}.

\begin{remark}[Whole space and bounded domains]\label{rem.domain}\rm
We believe that the whole space $\Omega=\R^d$ can be treated by using the
techniques of \cite{CHJZ21}, under
the assumption that a moment of $u^0$ is bounded, i.e.\
$\int_{\R^d}u_i^0(x)(1+|x|^2)^{\alpha/2}dx<\infty$ for a suitable $\alpha>0$.
Indeed, standard estimates show that $u_i^\eps(1+|x|^2)^{\alpha/2}$ is bounded
in $L^\infty(0,T;L^1(\R^d))$, where $(u_i^\eps)$ is a sequence of
approximating solutions. By the previous proof, $(\sqrt{u_i^\eps})$ 
is bounded in $L^2(0,T;H^1(\R^d))$, and estimate \eqref{ex.nablau} shows that
$(u_i^\eps)$ is bounded in $L^2(0,T;W^{1,q}(\R^d))$
with $q=(d+2)/(d+1)$. Since $W^{1,q}(\R^d)\cap L^1(\R^d;(1+|x|^2)^{\alpha/2})$ 
is compactly embedded in $L^r(\R^d)$ for $r<2q/(2-q)$ and $1\le q<2$
(by adapting the proof of \cite[Lemma 1]{CGZ20}), 
we can apply the Aubin--Lions lemma, concluding that, 
up to a subsequence, $u_i^\eps\to u$ strongly in $L^2(0,T;L^2(\R^d))$.

The case of bounded domains $\Omega\subset\R^d$ with Lipschitz boundary
$\pa\Omega$ seems to be more delicate. We assume no-flux boundary conditions
to recover the weak formulation \eqref{ex.wf}. The problem comes from the treatment
of the boundary integrals when integrating by parts. For instance, we need to
integrate by parts in $\na p_i[u]$ (see \eqref{ex.ibp}) and to control the integral
$$
  q_i[u](x) := \sum_{j=1}^n\int_{\pa\Omega} K_{ij}(x-y)u_j(y)\nu(y)dy,
$$
where $\nu$ is the exterior unit normal vector of $\pa\Omega$.
If $K_{ij}\in L^\infty(\R^d)$, we may estimate this integral by
$\|u_i\|_{L^1(\pa\Omega)}\le C\|u_i\|_{W^{1,1}(\Omega)}$. Consequently,
$\|q_i[u]\|_{L^\infty(\Omega)}\le C\sum_{j=1}^n\|u_j\|_{W^{1,1}(\Omega)}$.
The integral $q_i[u]$ appears in the weak formulation, for instance, as
$$
  \bigg|\sigma\int_\Omega\na u_i\cdot q_i[u]dx\bigg|
	\le C\sum_{j=1}^n\|\na u_j\|_{W^{1,1}(\Omega)}^2
	\le 2C\sum_{j=1}^n\|u_j\|_{L^1(\Omega)}\|\na u_j^{1/2}\|_{L^2(\Omega)}^2,
$$
and this integral cannot generally be absorbed by the corresponding term
in \eqref{ex.ei0} except if $\|u_j^0\|_{L^1(\Omega)}$ 
is sufficiently small.
\qed\end{remark}

\subsection*{Proof of Proposition \ref{prop.regul}.}

The proof of the $H^1(\dom)$ regularity requires an approximate scheme
different from that one used in the proof of Theorem \ref{thm.ex}. 
Given $u^{k-1}\in L^2(\dom;\R^n)$ with $u_i^{k-1}\ge 0$, 
we wish to find $u^k\in H^1(\dom;\R^n)$ such that
\begin{align}\label{ex.approx3}
  \frac{1}{\tau}&\int_{\dom}(u_i^k-u_i^{k-1})\phi_i dx
	+ \sigma\int_\dom\na u_i^k\cdot\na\phi_i dx
	+ \int_\dom\frac{(u_i^k)^+}{1+\delta(u_i^k)^+}\na p_i[u^k]\cdot\na\phi_i dx = 0,
\end{align}
for $\phi_i\in H^1(\dom)$, where $\delta>0$ and $z^+=\max\{0,z\}$.  
Since $\na K_{ij}\in L^{d+2}(\dom)$, $\na p_i[u^k]$ can be bounded in $L^{d+2}(\dom)$
in terms of the $L^1(\dom)$ norm of $u^k$. Thus, the last term on the left-hand side
is well defined. The existence of a solution 
to this discrete scheme is proved by a fixed-point argument, and the main step
is the derivation of uniform estimates. First, we observe that the test
function $(u_i^k)^-=\min\{0,u_i^k\}$ yields
$$
  \frac{1}{\tau}\int_{\dom}(u_i^k-u_i^{k-1})(u_i^k)^- dx
	+ \sigma\int_\dom|\na (u_i^k)^-|^2 dx
	= -\int_\dom\frac{(u_i^k)^+}{1+\delta(u_i^k)^+}\na p_i[u^k]\cdot\na (u_i^k)^- dx 
	= 0,
$$
and consequently, $(u_i^k)^-=0$ in $\dom$. Thus, $u_i^k\ge 0$ and we can remove
the plus sign in \eqref{ex.approx3}. Second, we use the test function $u_i^k$
in \eqref{ex.approx3} and sum over $i=1,\ldots,n$:
\begin{equation}\label{ex.sum}
  \frac{1}{\tau}\sum_{i=1}^n\int_\dom(u_i^k-u_i^{k-1})u_i^k dx
	+ \sigma\sum_{i=1}^n\int_\dom|\na u_i^k|^2 dx 
	= -\sum_{i=1}^n\int_\dom \frac{u_i^k}{1+\delta u_i^k}\na p_i[u^k]\cdot\na u_i^k dx.
\end{equation}
The first integral becomes
$$
  \sum_{i=1}^n\int_\dom(u_i^k-u_i^{k-1})u_i^k dx
	\ge \frac12\sum_{i=1}^n\int_\dom\big((u_i^k)^2 - (u_i^{k-1})^2\big)dx.
$$
The right-hand side in \eqref{ex.sum} is estimated by H\"older's inequality
and Young's convolution inequality:
\begin{align*}
  -\sum_{i=1}^n&\int_\dom\frac{u_i^k}{1+\delta u_i^k}\na p_i[u^k]\cdot\na u_i^k dx
	\le \sum_{i=1}^n\|u_i^k\|_{L^{2+4/d}(\dom)}\|\na p_i[u^k]\|_{L^{d+2}(\dom)}
	\|\na u_i^k\|_{L^2(\dom)} \\
	&\le C_K\sum_{i,j=1}^n\|u_i^k\|_{L^{2+4/d}(\dom)}\|u_j^k\|_{L^1(\dom)}
	\|\na u_i^k\|_{L^2(\dom)},
\end{align*}
where $C_K>0$ depends on the $L^{d+2}(\dom)$ norm of $\na K_{ij}$. Taking the
test function $\phi_i=1$ in \eqref{ex.approx3} shows that 
$\|u_i^k\|_{L^1(\dom)}=\|u_i^0\|_{L^1(\dom)}$ is uniformly bounded. 
This allows us to reduce the cubic expression on the right-hand side of the
previous inequality to a quadratic one. This is the key idea of the proof.
Combining the previous arguments, \eqref{ex.approx3} becomes
\begin{align*}
  \frac12\sum_{i=1}^n&\int_\dom(u_i^k)^2dx - \frac12\sum_{i=1}^n\int_\dom(u_i^{k-1})^2dx
	+ \tau\sigma\sum_{i=1}^n\int_\dom|\na u_i^k|^2dx \\
	&\le \tau C\sum_{i=1}^n\|u_i^k\|_{L^{2+4/d}(\dom)}\|\na u_i^k\|_{L^2(\dom)} \\
	&\le \frac{1}{2}\tau\sigma\sum_{i=1}^n\|\na u_i^k\|_{L^2(\dom)}^2
	+ \tau C\sum_{i=1}^n\|u_i^k\|_{L^{2+4/d}(\dom)}^2.
\end{align*}
The Gagliardo--Nirenberg and Poincar\'e--Wirtinger inequalities show that
\begin{align*}
  \|u_i^k\|_{L^{2+4/d}(\dom)}^2
	&\le C\|u_i^k\|_{H^1(\dom)}^{2\theta}\|u_i^k\|_{L^1(\dom)}^{2(1-\theta)} \\
  &\le C\big(\|\na u_i^k\|_{L^2(\dom)} + \|u_i^k\|_{L^1(\dom)}\big)^{2\theta}
	\|u_i^k\|_{L^1(\dom)}^{2(1-\theta)} \\
	&\le C(u^0)\|\na u_i^k\|_{L^2(\dom)}^{2\theta} + C(u^0),
\end{align*}
where $\theta=d(d+4)/(d+2)^2<1$. We deduce from Young's inequality that
for any $\eps>0$,
$$
  \|u_i^k\|_{L^{2+4/d}(\dom)}^2 \le \eps\|\na u_i^k\|_{L^2(\dom)}^2 + C(\eps).
$$
Therefore, choosing $\eps>0$ sufficiently small, we infer that
\begin{equation}\label{ex.u2}
  \frac12\sum_{i=1}^n\int_\dom(u_i^k)^2dx - \frac12\sum_{i=1}^n\int_\dom(u_i^{k-1})^2dx
	+ \frac14\tau\sigma\sum_{i=1}^n\int_\dom|\na u_i^k|^2dx \le C.
\end{equation}
This provides a uniform $H^1(\dom)$ estimate for $u^k$. 
Defining the fixed-point operator as a mapping from $L^2(\dom)$ to $L^2(\dom)$,
the compact embedding $H^1(\dom)\hookrightarrow L^2(\dom)$ implies the compactness 
of this operator (see \cite[Chapter 4]{Jue16} for details).
This shows that \eqref{ex.approx3} possesses a weak solution $u^k\in H^1(\dom)$.

In order to pass to the limit $(\delta,\tau)\to 0$, we need uniform estimates
for the piecewise constant in time functions $u_i^{(\tau)}$, using the
same notation as in the proof of Theorem \ref{thm.ex}. Estimate \eqref{ex.u2}
provides uniform bounds for $u_i^{(\tau)}$ in $L^\infty(0,T;L^2(\dom))$ and
$L^2(0,T;H^1(\dom))$. By the Gagliardo--Nirenberg inequality,
$(u_i^{(\tau)})$ is bounded in $L^{2+4/d}(Q_T)$. By Young's convolution
inequality, 
$$
  \sup_{t\in(0,T)}\|\na p_i[u^{(\tau)}(t)]\|_{L^\infty(\dom)}
	\le \sum_{j=1}^n\|\na K_{ij}\|_{L^{d+2}(\dom)}
	\sup_{t\in(0,T)}\|u_j^{(\tau)}\|_{L^q(\dom)} \le C,
$$
where $q=(d+2)/(d+1)$. Thus, 
$(\na p_i[u^{(\tau)}])$ is bounded in $L^\infty(0,T;L^\infty(\dom))$.
From these estimates, we can derive a uniform bound for the discrete time
derivative. Therefore, by the Aubin--Lions lemma \cite{DrJu12},
up to a subsequence, as $(\delta,\tau)\to 0$,
$$
  u_i^{(\tau)}\to u_i\quad\mbox{strongly in }L^2(Q_T),
$$
and this convergence even holds in $L^r(Q_T)$ for any $r<2+4/d$. We can show as
in the proof of Theorem \ref{thm.ex} that $p_i[u^{(\tau)}]\to p_i[u]$ a.e.\
and consequently, for a subsequence, $\na p_i[u^{(\tau)}]\rightharpoonup
\na p_i[u]$ weakly in $L^2(Q_T)$. We infer that
$$
  u_i^{(\tau)}\na p_i[u^{(\tau)}]\rightharpoonup u_i\na p_i[u]
	\quad\mbox{weakly in }L^1(Q_T).
$$
Omitting the details, it follows that $u=(u_1,\ldots,u_n)$ is a weak solution
to \eqref{1.eq}--\eqref{1.nonloc} satisfying $u_i\in L^2(0,T;H^1(\dom))$
for $i=1,\ldots,n$.

Next, we show the lower and upper bounds for $u_i$. Define $M(t)=M_0e^{\lambda t}$,
where $\lambda>0$ will be specified later. Since $\na K_{ij}\in L^\infty$,
we can apply the Young convolution inequality and
estimate $\na p_i[u]$ in $L^\infty(\dom)$ in terms of the $L^1(\dom)$ bounds
for $u_j$. Then, with the test function
$e^{-\lambda t}(u_i-M)^+(t)=e^{-\lambda t}\max\{0,(u_i-M)(t)\}$ 
in the weak formulation of \eqref{ex.approx3}, we deduce from
$$
  \pa_t u_i e^{-\lambda t}(u_i-M)^+
	= \frac12\pa_t\big\{e^{-\lambda t}[(u_i-M)^+]^2\big\}
	+ \frac{\lambda}{2}e^{-\lambda t}[(u_i-M)^+]^2 + \lambda e^{-\lambda t}M(u_i-M)^+,
$$
that
\begin{align*}
  \frac12&\int_\dom e^{-\lambda t}(u_i-M)^+(t)^2 dx 
	+ \sigma\int_0^t\int_\dom e^{-\lambda s}|\na(u_i-M)^+|^2 dxds \\
	&= -\int_0^t\int_\dom e^{-\lambda s}(u_i-M)\na p_i[u]\cdot
	\na(u_i-M)^+dxds - \frac{\lambda}{2}\int_0^t\int_\dom 
	e^{-\lambda s}[(u_i-M)^+]^2 dxds \\
	&\phantom{xx}{}- \int_0^t\int_\dom e^{-\lambda s}M\na p_i[u]\cdot\na(u_i-M)^+dxds
	- \lambda\int_0^t\int_\dom e^{-\lambda s}M(u_i-M)^+ dxds.
\end{align*}
We write $(u_i-M)\na(u_i-M)^+=\frac12\na[(u_i-M)^+]^2$ and integrate by parts
in the first and third terms of the right-hand side:
\begin{align*}
  \frac12&\int_\dom e^{-\lambda s}(u_i-M)^+(t)^2 dx 
	+ \sigma\int_0^t\int_\dom e^{-\lambda s}|\na(u_i-M)^+|^2 dxds \\
	&\le \frac12\big(\|\Delta p_i[u]\|_{L^\infty(0,T;L^\infty(\dom))}-\lambda\big)
	\int_0^t\int_\dom e^{-\lambda s}[(u_i-M)^+]^2 dxds \\
	&\phantom{xx}{}+ \big(\|\Delta p_i[u]\|_{L^\infty(0,T;L^\infty(\dom))}
	-\lambda\big)\int_0^t\int_\dom e^{-\lambda s}M(u_i-M)^+ dxds.
\end{align*}
By Young's convolution inequality and the regularity assumptions on $K_{ij}$,
$$
	\|\Delta p_i[u]\|_{L^\infty(0,T;L^\infty(\dom))} 
	\le C\sum_{j=1}^n\|u_j\|_{L^\infty(0,T;L^1(\dom))} \le C_0.
$$
Therefore, choosing $\lambda\ge C_0$, it follows that
$$
  \int_\dom e^{-\lambda t}(u_i-M)^+(t)^2 dx \le 0,
$$ 
and we infer that $e^{-\lambda t}(u_i-M)^+(t)=0$ and 
$u_i(t)\le M(t)=M_0 e^{\lambda t}$ for $t>0$.
The inequality $u_i(t)\ge m(t):=m_0e^{-\lambda t}$ is proved in the same way, using
the test function $e^{-\lambda t}(u_i-m)^-=e^{-\lambda t}\min\{0,u_i-m\}$.

\begin{remark}\rm\label{rem.kernel}
Proposition \ref{prop.regul} holds true for functions
$K_{ij}(x,y)$ that are not convolution kernels. We need the
regularity $\na_x K_{ij}\in L_y^\infty L_x^{d+2}\cap L_x^\infty L_y^{d+2}$
to apply the Young inequality for kernels; see \cite[Theorem 0.3.1]{Sog93}.
For the lower and upper bounds of the solution, we additionally need the 
regularity $\na_x K_{ij}$, $\Delta_x K_{ij} \in L_x^\infty L_y^\infty$.
\qed\end{remark}

%%%%%%%%%%%%%%%%%%%%%%%%%%%%%%%%%%%%%%%%%%%%%%%%%%%%%%%%%%%%%%%%%%%%%%%%%%%%%%%

\section{Weak-strong uniqueness for the nonlocal system}\label{sec.wsn}

In this section, we prove Theorem \ref{thm.ws}. 
Let $u$ be a nonnegative weak solution
and $v$ be a positive ``strong'' solution to \eqref{1.eq}--\eqref{1.nonloc},
i.e., $v=(v_1,\ldots,v_n)$ is a weak solution to \eqref{1.eq}--\eqref{1.nonloc} 
satisfying 
$$
  0<c\le v_i\le C\quad\mbox{in }Q_T,
	\quad \pa_t v_i\in L^2(0,T;H^{-1}(\dom)), 
	\quad v_i\in L^\infty(0,T;W^{1,\infty}(\dom)).
$$
In particular, we have $\log v_i$, $\na\log v_i\in L^\infty(0,T;L^\infty(\dom))$.
Then, for $0 < \eps < 1$, we define the regularized relative entropy density
$$
  h_{\eps}(u|v) = \sum_{i=1}^n\pi_i\big((u_i+\eps)(\log(u_i+\eps)-1)
	- (u_i+\eps)\log v_i + v_i\big),
$$
and the associated relative entropy
$$
  H_{\eps}(u|v) = \int_\dom h_{\eps}(u|v)dx.
$$

{\em Step 1: Preparations.} We compute
$$
  \frac{\pa h_{\eps}}{\pa u_i}(u|v) = \pi_i\log(u_i+\eps) - \pi_i\log v_i, \quad
  \frac{\pa h_{\eps}}{\pa v_i}(u|v) = -\pi_i\frac{u_i+\eps}{v_i} + \pi_i.
$$ 
The second function is an admissible test function for the weak formulation of
\eqref{1.eq}, satisfied by $v_i$, since $\na u_i\in L^2(Q_T)$ and
$\na v_i\in L^\infty(Q_T)$.  
Strictly speaking, the first function is not an admissible test function for
the weak formulation of \eqref{1.eq}, satisfied by $u_i$, since it needs test functions in $W^{1,d+2}(\dom)$. However, the nonlocal term becomes with this test function
$$
  \int_\dom\int_\dom K_{ij}(x-y)\na u_j(y)\cdot\frac{\na u_i(x)}{u_i(x)+\eps}dxdy,
$$
which is finite since $K_{ij}$ is essentially bounded and 
$\na u_i\cdot\na u_j\in L^1(Q_T)$. Thus, by a suitable approximation argument, 
the following computation can be made rigorous. We find that
\begin{align*}
  \frac{d}{dt}H_{\eps}(u|v) &= \sum_{i=1}^n\bigg(
	\bigg\langle\pa_t u_i,\frac{\pa h_{\eps}}{\pa u_i}(u|v)\bigg\rangle
  + \bigg\langle\pa_t v_i,\frac{\pa h_{\eps}}{\pa v_i}(u|v)\bigg\rangle\bigg) \\
	&= -\sigma\sum_{i=1}^n\int_\dom\bigg(\na u_i\cdot\na
	\frac{\pa h_{\eps}}{\pa u_i}(u|v) + \na v_i\cdot\na 
	\frac{\pa h_{\eps}}{\pa v_i}(u|v)\bigg)dx \\
	&\phantom{xx}{} -\sum_{i=1}^n\int_\dom\bigg(u_i\na p_i[u]\cdot
	\na\frac{\pa h_{\eps}}{\pa u_i}(u|v)
	+ v_i\na p_i[v]\cdot\na \frac{\pa h_{\eps}}{\pa v_i}(u|v)\bigg)dx \\
	&= -\sigma\sum_{i=1}^n\int_\dom\pi_i\bigg|\frac{\na u_i}{\sqrt{u_i+\eps}}
	- \sqrt{u_i+\eps}\frac{\na v_i}{v_i}\bigg|^2dx \\
	&\phantom{xx}{}-\sum_{i=1}^n\int_\dom\pi_i
	\bigg(\frac{u_i}{u_i+\eps}\na p_i[u]\cdot\na u_i
	- \frac{u_i}{v_i}\na p_i[u]\cdot\na v_i - \na p_i[v]\cdot\na u_i \\
	&\phantom{xx}{}+ \frac{u_i+\eps}{v_i}\na p_i[v]\cdot\na v_i\bigg)dx.
\end{align*}
The first integral is nonpositive. Thus, an integration over $(0,t)$ gives
\begin{align}
  H_{\eps}&(u(t)|v(t)) - H_\eps(u(0)|v(0)) \nonumber \\ 
	&\le -\sum_{i=1}^n\int_0^t\int_\dom\pi_i
	(u_i+\eps)\na(p_i[u]-p_i[v])\cdot\na\log\frac{u_i+\eps}{v_i} dxds \nonumber \\
	&\phantom{xx}{}+ \eps\sum_{i=1}^n\int_0^t\int_\dom\pi_i\na p_i[u]\cdot
	\na\log\frac{u_i+\eps}{v_i} dxds =: I_1 + I_2. \label{ws.dHdt}
\end{align}

{\em Step 2: Estimation of $I_1$ and $I_2$.}
Inserting the definition of $p_i$,
\begin{align*}
  \na(p_i[u]-p_i[v])(x) &= \sum_{j=1}^n\int_\dom K_{ij}(x-y)\na(u_j-v_j)(y)dy \\
	&= \sum_{j=1}^n\int_\dom K_{ij}(x-y)\bigg(
	(u_j+\eps)(y)\na\log\frac{u_j+\eps}{v_j}(y) \\
	&\phantom{xx}{}+ (u_j-v_j)(y)\na\log v_j(y) + \eps\na\log v_j(y)\bigg)dy,
\end{align*}
leads to
\begin{align*}
  I_1 &= -\sum_{i,j=1}^n\int_0^t\int_\dom\int_\dom\pi_i K_{ij}(x-y)
	\bigg((u_i+\eps)(x)(u_j+\eps)(y)
	\na\log\frac{u_j+\eps}{v_j}(y) \\
	&\phantom{xx}{}\times\na\log\frac{u_i+\eps}{v_i}(x)
	+ (u_i+\eps)(x)(u_j-v_j)(y)\na\log v_j(y)\cdot\na\log\frac{u_i+\eps}{v_i}(x)
	\bigg)dxdyds \\
	&\phantom{xx}{}
	- \eps\sum_{i,j=1}^n\int_0^t\int_\dom\int_\dom\pi_i K_{ij}(x-y)
	(u_i+\eps)(x)\na\log v_j(y)\cdot\na\log\frac{u_i+\eps}{v_i}(x)dxdyds \\
	&=: I_{11} + I_{12}.
\end{align*}
Setting
$$
  U_i = (u_i+\eps)\na\log\frac{u_i+\eps}{v_i}, \quad
	V_i = \frac12(u_i-v_i)\na\log v_i,
$$
we can formulate the first integral as
\begin{align*}
  I_{11} &= -\sum_{i,j=1}^n\int_0^t\int_\dom\int_\dom\pi_i K_{ij}(x-y)
	\big(U_i(x)\cdot U_j(y) + 2U_i(x)\cdot V_j(y)\big)dxdyds \\
	&= -\sum_{i,j=1}^n\int_0^t\int_\dom\int_\dom\pi_i K_{ij}(x-y)
	(U_i+V_i)(x)\cdot(U_j+V_j)(y)dxdyds \\
	&\phantom{xx}{}
	+ \sum_{i,j=1}^n\int_0^t\int_\dom\int_\dom\pi_i K_{ij}(x-y) 
	V_i(x)\cdot V_j(y)dxdyds \\
	&\le \frac14\sum_{i,j=1}^n\int_0^t\int_\dom\int_\dom\pi_i K_{ij}(x-y)
	(u_i-v_i)(x)(u_j-v_j)(y)\na\log v_i(x)\cdot\na\log v_j(y)dxdyds \\
	&\le \frac14\max_{i,j=1,\ldots,n}\|\pi_i K_{ij}\|_{L^\infty(\dom)}
	\max_{k=1,\ldots,n}\|\na\log v_k\|_{L^\infty(Q_T)}^2 \\
	&\phantom{xx}{}\times\sum_{i,j=1}^n
	\int_0^t\int_\dom|(u_i-v_i)(x)|dx\int_\dom|(u_j-v_j)(y)|dyds \\
	&\le C\sum_{i=1}^n\int_0^t\bigg(\int_\dom|u_i-v_i|dx\bigg)^2 ds,
\end{align*}
using the symmetry and positive definiteness of $\pi_i K_{ij}$ as well as
the regularity assumptions on $K_{ij}$ and $\na\log v_i$. 
The second integral $I_{12}$ is estimated as
\begin{align*}
  I_{12} &= -\eps\sum_{i,j=1}^n\int_0^t\int_\dom\int_\dom\pi_i K_{ij}(x-y)
	\na\log v_j(y)\cdot\big(\na u_i - (u_i+\eps)\na\log v_i\big)(x)dxdyds \\
	&\le \eps C\sum_{i,j=1}^n\|\na\log v_j\|_{L^\infty(Q_T)}
	\int_0^t\int_\dom\big(|\na u_i|+(u_i+1)|\na \log v_i|\big)dxds \\
	&\le \eps C\sum_{i=1}^n\big(\|\na u_i\|_{L^1(Q_T)}+\|u_i\|_{L^1(Q_T)} + 1\big) 
	\le \eps C.
\end{align*}
We conclude that
$$
  I_1 \le C\sum_{i=1}^n\int_0^t\bigg(\int_\dom|u_i-v_i|dx\bigg)^2 ds
	+ \eps C.
$$

It remains to estimate $I_2$. Here we need the improved regularity
$\na u_i\in L^2(Q_T)$. Inserting the definition of $p_i[u]$, we have
$$
  I_2 = \eps\sum_{i,j=1}^n\int_0^t\int_\dom\int_\dom\pi_i K_{ij}(x-y)
	\na u_j(y)\cdot\na\log\frac{u_i+\eps}{v_i}(x)dxdyds.
$$
Since
$$
  \eps|\na u_j(y)\cdot\na\log(u_i+\eps)(x)|
	= 2\eps\bigg|\na u_j(y)\cdot\frac{\na\sqrt{u_i+\eps}}{\sqrt{u_i+\eps}}(x)\bigg|
	\le 2\sqrt{\eps}|\na u_j(y)||\na\sqrt{u_i(x)}|,
$$
we find that
$$
  I_2\le C\sum_{i,j=1}^n\big(\eps\|\na u_j\|_{L^1(Q_T)}
	+ \sqrt{\eps}\|\na u_j\|_{L^2(Q_T)}\|\na\sqrt{u_i}\|_{L^2(Q_T)}\big) 
	\le \sqrt{\eps}C.
$$

We summarize the estimates for $I_1$ and $I_2$ and conclude from \eqref{ws.dHdt} that
\begin{equation}\label{ws.Heps}
  H_\eps(u(t)|v(t)) - H_\eps(u(0)|v(0))
	\le C\sum_{i=1}^n\int_0^t\bigg(\int_\dom|u_i-v_i|dx\bigg)^2 ds
	+ \sqrt{\eps}C.
\end{equation}

{\em Step 3: Limit $\eps\to 0$.}
We perform first the limit in $H_\eps(u(t)|v(t))$. 
Since $u_i\in L^2(0,T;$ $H^1(\dom))\cap H^1(0,T;H^{-1}(\dom))\hookrightarrow
C^0([0,T];L^2(\dom))$, we have 
$$
  \big|(u_i+\eps)(\log(u_i+\eps)-1)\big|
	\le u_i(\log u_i+1) + C\in L^\infty(0,T;L^1(\dom)).
$$
Therefore, by dominated convergence, as $\eps\to 0$,
$$
  \sum_{i=1}^n\int_\dom\pi_i(u_i(t)+\eps)(\log(u_i(t)+\eps)-1)dx
	\to \sum_{i=1}^n\int_\dom\pi_i u_i(t)(\log u_i(t)-1)dx,
$$
and this convergence holds for a.e.\ $t\in(0,T)$. Furthermore, in view of the
bound for $\log v_i$,
$$
  \sum_{i=1}^n\pi_i(-(u_i+\eps)\log v_i+v_i)
	\le C(v)\bigg(\sum_{i=1}^n u_i+1\bigg)\in L^\infty(0,T;L^1(\dom)),
$$
and we can again use dominated convergence:
$$
  \sum_{i=1}^n\int_\dom\pi_i\big(-(u_i(t)+\eps)\log v_i(t)+v_i(t)\big)dx
	\to \sum_{i=1}^n\int_\dom\pi_i\big(-u_i(t)\log v_i(t)+v_i(t)\big)dx.
$$
This shows that for a.e.\ $t\in(0,T)$,
\begin{align*}
  & H_\eps(u(t)|v(t)) \to H(u(t)|v(t))\quad\mbox{as }\eps\to 0,\mbox{ where} \\
  & H(u|v) = \sum_{i=1}^n\int_\dom\pi_i\big(u_i(\log u_i-1) - u_i\log v_i+v_i\big)dx,
\end{align*}
and $H_\eps(u(0)|v(0))=H_\eps(u^0|u^0)\to 0$.
Then we deduce from \eqref{ws.Heps} in the limit $\eps\to 0$ that
\begin{equation}\label{ws.aux}
  H(u(t)|v(t)) \le  C\sum_{i=1}^n\int_0^t\|u_i-v_i\|_{L^1(\dom)}^2 ds.
\end{equation}

Taking the test function $\phi_i=1$ in the weak formulation of \eqref{1.eq},
we find that $\int_\dom u_i^0dx = \int_\dom u_i(t)dx$ for all $t>0$. 
Since $u$ and $v$ have the same initial data, it follows that 
$\int_\dom u_i(t)dx=\int_\dom v_i(t)dx$ for all $t>0$. Thus,
by the classical Csisz\'ar--Kullback--Pinsker inequality \cite[Theorem A.2]{Jue16},
we have
\begin{align*}
  H(u|v) &= \sum_{i=1}^n\int_\dom\pi_i u_i\log\frac{u_i}{v_i}dx 
	+ \sum_{i=1}^n\int_\dom\pi_i(v_i-u_i)dx \\
	&= \sum_{i=1}^n\int_\dom\pi_i u_i\log\frac{u_i}{v_i}dx
	\ge C(u^0)\sum_{i=1}^n\|u_i-v_i\|_{L^1(\dom)}^2.
\end{align*}
We infer from \eqref{ws.aux} that
$$
  \sum_{i=1}^n\|(u_i-v_i)(t)\|_{L^1(\dom)}^2 
	\le C\int_0^t\sum_{i=1}^n\|u_i-v_i\|_{L^1(\dom)}^2 ds.
$$
Gronwall's inequality implies that $\|(u_i-v_i)(t)\|_{L^1(\dom)}=0$ and hence
$u_i(t)=v_i(t)$ in $\dom$ for a.e.\ $t>0$ and $i=1,\ldots,n$.

%%%%%%%%%%%%%%%%%%%%%%%%%%%%%%%%%%%%%%%%%%%%%%%%%%%%%%%%%%%%%%%%%%%%%%%%%%%%%%%

\section{Localization limit}\label{sec.loc}

We prove Theorem \ref{thm.loc}. Let $u^\eps$ be the nonnegative weak solution
to \eqref{1.eq}--\eqref{1.nonloc} with kernel \eqref{2.K}, constructed in Theorem
\ref{thm.ex}. The entropy inequalities \eqref{ex.H1} and \eqref{ex.H2} as
well as the proof of Theorem \ref{thm.ex} show that
all estimates are independent of $\eps$.
(More precisely, the right-hand side of \eqref{ex.H2} depends on $K_{ij}^\eps$,
but in view of \cite[Theorem 4.22]{Bre10}, the right-hand side can be bounded
uniformly in $\eps$.)
Therefore, for $i=1,\ldots,n$ 
(see \eqref{ex.ulog}--\eqref{ex.nablau}, \eqref{ex.up}--\eqref{ex.time}),
\begin{align*}
  \|u_i^\eps\log u_i^\eps\|_{L^\infty(0,T;L^1(\dom))}
	+ \|u_i^\eps\|_{L^{1+2/d}(Q_T)} 
	+ \|u_i^\eps\|_{L^q(0,T;W^{1,q}(\dom))} &\le C, \\
  \|(u_i^\eps)^{1/2}\|_{L^{2}(0,T;H^1(\dom))} 
	+ \|\pa_t u_i^\eps\|_{L^q(0,T;W^{1,d+2}(\dom)')} 
	+ \|(u_i^\eps)^{1/2}\na p_i^\eps[u^\eps]\|_{L^2(Q_T)} &\le C,
\end{align*}
where $C>0$ is independent of $\eps$, $q=(d+2)/(d+1)$, and
$p_i^\eps[u_i^\eps]=\sum_{j=1}^n\int_\dom K_{ij}^\eps(x-y)u_j^\eps(y)dy$. 
We infer from the Aubin--Lions
lemma in the version of \cite{BCJ20,DrJu12} that there exists a subsequence 
(not relabeled) such that, as $\eps\to 0$,
\begin{equation}\label{loc.conv0}
  u^\eps_i\to u_i\quad\mbox{strongly in }L^2(0,T;L^{d/(d-1)}(\dom)), \quad
	i=1,\ldots,n,
\end{equation}
if $d \geq 2$ and strongly in $L^2(0,T;L^r(\dom))$ for any $r<\infty$ if $d=1$.
Moreover,
\begin{align}
  \na u_i^\eps \rightharpoonup \na u_i &\quad\mbox{weakly in }L^q(Q_T), 
	\quad i=1,\ldots,n, \label{loc.conv1} \\
  \pa_t u_i^\eps\rightharpoonup\pa_t u_i
	&\quad\mbox{weakly in }L^q(0,T;W^{1,d+2}(\dom)'), \label{loc.conv2} \\
	(u_i^\eps)^{1/2}\na p_i^\eps[u^\eps]\rightharpoonup z_i 
	&\quad\mbox{weakly in }L^2(Q_T), \label{loc.conv3}
\end{align}
where $z_i\in L^2(Q_T)$ for $i=1,\ldots,n$.

As in Section \ref{sec.ex}, the main difficulty is the identification of
$z_i$ with $u_i^{1/2}\na p_i[u]$, where $p_i[u]:=\sum_{j=1}^n a_{ij}\na u_j$.
Since the kernel functions also depend on $\eps$,
the proof is different from that one in Section \ref{sec.ex}. We claim that
\begin{equation}\label{loc.nap}
  \na p_i^\eps[u^\eps]\rightharpoonup \na p_i[u]
	\quad\mbox{weakly in }L^q(Q_T).
\end{equation}
Indeed, let $\phi\in L^{q'}(Q_T;\R^n)$, where $q'=d+2$ satisfies $1/q+1/q'=1$.
We compute
\begin{align*}
  \bigg|\int_0^T&\int_\dom(\na p_i^\eps[u^\eps] - \na p_i[u])\cdot\phi dxdt\bigg| \\
	&= \bigg|\sum_{j=1}^n\int_0^T\int_\dom\bigg(\int_\dom K_{ij}^\eps(x-y)
	\na u_j^\eps(y,t)dy\bigg)\cdot\phi(x,t)dxdt \\
	&\phantom{xx}{}- 
	\sum_{j=1}^n\int_0^T\int_\dom a_{ij}\na u_j(y,t)\cdot\phi(y,t)dydt\bigg| \\
  &\le \sum_{j=1}^n\bigg|\int_0^T\int_\dom\bigg(\int_\dom 
	K_{ij}^\eps(x-y)\phi(x,t)dx
	- a_{ij}\phi(y,t)\bigg)\cdot\na u_j^\eps(y,t)dydt\bigg| \\
	&\phantom{xx}{}+ \sum_{j=1}^n
	a_{ij}\bigg|\int_0^T\int_\dom\na(u_j^\eps-u_j)(y,t)\cdot\phi(y,t)dydt\bigg| \\
	&\le \sum_{j=1}^n\bigg\|\int_\dom K_{ij}^\eps(\cdot-y)\phi(y)dy 
	- a_{ij}\phi\bigg\|_{L^{q'}(Q_T)}\|\na u_j^\eps\|_{L^q(Q_T)} \\
	&\phantom{xx}{}+ \sum_{j=1}^n
	a_{ij}\bigg|\int_0^T\int_\dom\na(u_j^\eps-u_j)(y,t)\cdot\phi(y,t)dydt\bigg|.
\end{align*}
Since $B$ has compact support in $\R$, we can apply the proof of 
\cite[Theorem 4.22]{Bre10} to infer that the first term on the right-hand
side, formulated as the convolution $K_{ij}^\eps*\phi-a_{ij}\phi$
(slightly abusing the notation),
converges to zero strongly in $L^{q'}(\R^d)$ as $\eps\to 0$.
Thus, taking into account the weak convergence \eqref{loc.conv1}, 
convergence \eqref{loc.nap} follows. 

Because of the convergences \eqref{loc.conv0}, \eqref{loc.conv3}, and
\eqref{loc.nap}, we can apply Lemma \ref{lem.ident} in Appendix \ref{aux} to infer
that $z_i=u_i^{1/2}\na p_i[u]$. Therefore,
$$
  u_i^\eps\na p_i^\eps[u^\eps]\rightharpoonup u_i\na p_i[u]\quad
	\mbox{weakly in }L^1(Q_T).
$$
Estimate \eqref{loc.conv3} shows
that the convergence holds in $L^q(Q_T)$. This convergence as well as 
\eqref{loc.conv1} and \eqref{loc.conv2} allow us to perform the
limit $\eps\to 0$ in the weak formulation of \eqref{1.eq}, proving that $u$
solves \eqref{1.eq} and \eqref{1.loc}. 

%%%%%%%%%%%%%%%%%%%%%%%%%%%%%%%%%%%%%%%%%%%%%%%%%%%%%%%%%%%%%%%%%%%%%%%%%%%%%%%%

\begin{appendix}
\section{Auxiliary results}\label{aux}

We recall the Young convolution inequality (the proof in \cite[Theorem 4.33]{Bre10}
also applies to the torus).

\begin{lemma}[Young's convolution inequality]\label{lem.young}
Let $1\le p\le q\le\infty$ be such that $1+1/q=1/p+1/r$
and let $K\in L^r(\dom)$ (extended periodically to $\R^d$). 
Then for any $v\in L^p(\dom)$,
$$
  \bigg\|\int_\dom K(\cdot-y)v(y)dy\bigg\|_{L^q(\dom)}
	\le \|K\|_{L^r(\dom)}\|v\|_{L^p(\dom)}.
$$
\end{lemma}

The next result is a consequence of Vitali's lemma
and is well known. We recall it for the convenience of the reader.

\begin{lemma}\label{lem.conv}
Let $\Omega\subset\R^d$ $(d\ge 1)$ be a bounded domain, $1<p<\infty$, and
$u_\eps$, $u\in L^1(\Omega)$ be such that $(u_\eps)$ is bounded in $L^p(\Omega)$
and $u_\eps\to u$ a.e.\ in $\Omega$. Then $u_\eps\to u$
strongly in $L^r(\Omega)$ for all $1\le r<p$.
\end{lemma}

\begin{proof}
We have for any $M>0$,
$$
  \int_{\{u_\eps\ge M\}}|u_\eps|^rdx 
	= \int_{\{u_\eps\ge M\}}|u_\eps|^p|u_\eps|^{-(p-r)}dx
	\le M^{-(p-r)}\int_\Omega|u_\eps|^p dx \le CM^{-(p-r)}\to 0,
$$
as $M\to\infty$. Thus, $(u_\eps)$ is uniformly integrable. Since
convergence a.e.\ implies convergence in measure, we conclude
with Vitali's convergence theorem.
\end{proof}

The following lemma specifies conditions under which the limit of the product 
of two converging sequences can be identified.

\begin{lemma}\label{lem.ident}
Let $p>1$ and let $u_\eps\ge 0$, $u_\eps\to u$ strongly in $L^p(\dom)$,
$v_\eps\rightharpoonup v$ weakly in $L^p(\dom)$, and
$u_\eps v_\eps\rightharpoonup w$ weakly in $L^p(\dom)$ as $\eps\to 0$.
Then $w=uv$.
\end{lemma}

The lemma is trivial if $p\ge 2$. We apply it in Section \ref{sec.ex}
with $1<p<2$. Note that the strong convergence of $(u_\eps)$ cannot be replaced
by weak convergence. A simple counter-example is given by
$u_\eps(x)=\exp(2\pi\mathrm{i}x/\eps)\rightharpoonup 0$,
$v_\eps(x)=\exp(-2\pi\mathrm{i}x/\eps)\rightharpoonup 0$ weakly in $L^2(-1,1)$,
but $u_\eps v_\eps\equiv 1\neq 0\cdot 0$.

\begin{proof}
We define the truncation function $T_1\in C^2([0,\infty))$ satisfying
$T_1(s)=s$ for $0\le s\le 1$, $T_1(s)=2$ for $s>3$, and $T_1$ is nondecreasing
and concave in the interval $[1,3]$. 
Furthermore, we set $T_k(s)=kT_1(s/k)$ for $s\ge 0$
and $k\in\N$. The strong convergence of $(u_\eps)$ implies the existence of
a subsequence that is not relabeled such that $u_\eps\to u$ a.e. Hence,
$T_k(u_\eps)\to T_k(u)$ a.e. Since $T_k$ is bounded for fixed $k\in\N$,
we conclude by dominated convergence that $T_k(u_\eps)\to T_k(u)$ strongly
in $L^r(\dom)$ for any $r<\infty$. Because of the uniqueness of the limit,
the convergence holds for the whole sequence. 
Thus, $T_k(u_\eps)v_\eps\rightharpoonup T_k(u)v$ weakly in $L^1(\dom)$.
Writing $\overline{z_\eps}$ for the weak limit of a sequence $(z_\eps)$ (if it
exists), this result means that $\overline{T_k(u_\eps)v_\eps}=T_k(u)v$
and the assumption translates to $\overline{u_\eps v_\eps}=w$.
Consequently, $w-T_k(u)v=\overline{(u_\eps-T_k(u_\eps))v_\eps}$.
Then we can estimate
\begin{align*}
  \|w-T_k(u)v\|_{L^1(\dom)}
	&\le \sup_{0<\eps<1}\int_\dom |u_\eps-T_k(u_\eps)||v_\eps| dx
	\le C\sup_{0<\eps<1}\int_{\{|u_\eps|>k\}}|u_\eps| |v_\eps| dx \\
	&\le \frac{C}{k^{p-1}}\int_{\{|u_\eps|>k\}}|u_\eps|^p |v_\eps| dx
	\le \frac{C}{k^{p-1}}\int_\dom |u_\eps|^p(1 + v_\eps|^p)dx \le \frac{C}{k^{p-1}}.
\end{align*}
This shows that $T_k(u)v\to w$ strongly in $L^1(\dom)$ and (for a subsequence)
a.e.\ as $k\to\infty$.
Since $T_k(u)v=uv$ in $\{|u|\le k\}$ for any $k\in\N$ and
$\operatorname{meas}\{|u|>k\}\le\|u\|_{L^1(\dom)}/k\to 0$, we infer in the
limit $k\to\infty$ that $w=uv$ in $\dom$.
\end{proof}

%%%%%%%%%%%%%%%%%%

\section{Local cross-diffusion system}\label{local}

The existence of global weak solutions to the local system \eqref{1.eq}
and \eqref{1.loc} in any bounded polygonal domain was shown in \cite{JuZu20} 
by analyzing a finite-volume scheme. For completeness, we state the assumptions and the
theorem and indicate how the result can be proved using the techniques of
Section \ref{sec.ex}.
We assume that $\Omega\subset\R^d$ ($d\ge 1$) is a bounded domain,
$T>0$, and $u^0\in L^2(\Omega)$ satisfies $u_i^0\ge 0$ in $\Omega$ 
for $i=1,\ldots,n$. We set $Q_T=\Omega\times(0,T)$.

\begin{theorem}[Existence for the local system]
Let $\sigma>0$, $a_{ij}\ge 0$, and let the matrix $(u_i a_{ij})$ 
be positively stable for all $u_i>0$, 
$i=1,\ldots,n$. Assume that there exist $\pi_1,\ldots,\pi_n>0$ such that
$\pi_i a_{ij} = \pi_j a_{ji}$ for all $i,j=1,\ldots,n$.
Then there exists a global weak solution to \eqref{1.eq} and 
\eqref{1.loc}, satisfying $u_i\ge 0$ in $Q_T$ and
\begin{align*}
  u_i\in L^2(0,T;H^1(\Omega))\cap L^{2+4/d}(Q_T), \quad 
	\pa_t u_i\in L^q(0,T;W^{-1,q}(\Omega)),
\end{align*}
for $i=1,\ldots,n$, where $q=(d+2)/(d+1)$. 
The initial datum in \eqref{1.eq} is satisfied in the sense of $W^{-1,q}(\Omega)$.
Moreover, the following entropy inequalities are satisfied:
\begin{equation}\label{eiloc}
\begin{aligned}
  & \frac{dH_1}{dt} + 4\sigma\sum_{i=1}^n\int_\Omega\pi_i|\na\sqrt{u_i}|^2 dx
	+ \alpha\sum_{i=1}^n\int_\Omega |\na u_i|^2 dx \le 0, \\
  & \frac{dH_2^0}{dt} + \sum_{i=1}^n\int_\Omega\pi_i u_i|\na p_i[u]|^2 dx
	+ \alpha\sigma\sum_{i=1}^n\int_\Omega |\na u_i|^2 dx \le 0,
\end{aligned}
\end{equation}
where $\alpha>0$ is the smallest eigenvalue of $(\pi_i a_{ij})$
and $H_2^0(u):=\frac12\sum_{i,j=1}^n\int_\Omega\pi_i a_{ij}u_i u_jdx\ge 0$.
\end{theorem}

We call a matrix {\em positively stable} if all eigenvalues have a positive real part.
This condition means that \eqref{1.eq} is parabolic in the sense of Petrovskii,
which is a minimal condition to ensure the local solvability \cite{Ama93}. 
Inequalities \eqref{1.H1}--\eqref{1.H2} and \eqref{eiloc} reveal a link
between the entropy structures of the nonlocal and local systems. 
This link was explored recently in detail for related systems in \cite{DiMo21}.

\begin{proof}
If $\Omega$ is the torus, the theorem is a consequence of the localization limit 
(Theorem \ref{thm.loc}). If $\Omega$ is a bounded domain,
the result can be proved by using the techniques of the proof
of Theorem \ref{thm.ex}. In fact, the proof is simpler since the problem is local.
The entropy identities are (formally)
\begin{equation}\label{app.eiloc}
\begin{aligned}
  & \frac{dH_1}{dt} + 4\sigma\sum_{i=1}^n\int_\Omega\pi_i|\na\sqrt{u_i}|^2 dx
	= -\sum_{i,j=1}^n\int_\Omega\pi_i a_{ij}\na u_i\cdot\na u_j dx, \\
  & \frac{dH_2^0}{dt} + \sum_{i=1}^n\int_\Omega\pi_i u_i|\na p_i[u]|^2 dx
	= -\sigma\sum_{i,j=1}^n\int_\Omega\pi_i a_{ij}\na u_i\cdot\na u_j dx.
\end{aligned}
\end{equation}
We claim that the matrix $(\pi_i a_{ij})$
is positive definite. Let $A_1:=\operatorname{diag}(u_i/\pi_i)$ and
$A_2:=(\pi_i a_{ij})$. Then $A_1$ is symmetric and positive definite;
by our assumptions, $A_2$ is symmetric and 
$A_1A_2=(u_i a_{ij})$ is positively stable. 
Therefore, by \cite[Prop.~3]{ChJu21}, $A_2$ is positive definite. We infer that
the right-hand sides in \eqref{app.eiloc} are nonpositive, and we derive
estimates for an approximate family of $u_i$ in $L^\infty(0,T;L^2(\Omega))$
and $L^2(0,T;H^1(\Omega))$. By the Gagliardo--Nirenberg inequality, this yields
bounds for $u_i$ in $L^{2+4/d}(Q_T)$. Consequently, $u_i\na p_i[u]$ is bounded
in $L^q(Q_T)$, where $q=(d+2)/(d+1)$ (we can even choose $q=4(d+2)/(3d+4)$),
and $\pa_t u_i$ is bounded in $L^q(0,T;W^{-1,q}(\Omega))$. These estimates
are sufficient to deduce from the Aubin--Lions lemma the
relative compactness for the approximate family of $u_i$ in $L^2(Q_T)$. 
The limit in the approximate problem, similar to \eqref{ex.approx}, shows that 
the limit satisfies \eqref{1.eq} and \eqref{1.loc}.
Finally, using the lower semicontinuity of convex functions and the
norm, the weak limit in the entropy inequalities leads to \eqref{eiloc}.
\end{proof}

\end{appendix}

%%%%%%%%%%%%%%%%%%%%%%%%%%%%%%%%%%%%%%%%%%%%%%%%%%%%%%%%%%%%%%%%%%%%%%%%%%%%%%%%

\end{document}